\documentclass[10pt]{amsart}

\usepackage{amsmath, amsfonts, amssymb, amsthm, mathdots, mathrsfs, dsfont}

\usepackage{tikz}

\newtheorem{theorem}{Theorem}[section]

\newtheorem{proposition}[theorem]{Proposition}
\newtheorem{conjecture}[theorem]{Conjecture}
\newtheorem{corollary}[theorem]{Corollary}
\newtheorem{lemma}[theorem]{Lemma}
\newtheorem{remark}[theorem]{Remark}
\newtheorem{assumption}[theorem]{Assumption}

\numberwithin{equation}{section}

\newcommand\otheta{\overline{\theta}}
\newcommand\opsi{\overline{\psi}}

\begin{document}

\title[On rationality and holomorphy of Langlands-Shahidi $L$-functions]{On the rationality and holomorphy \\ of Langlands-Shahidi $L$-functions \\ over function fields}


\author{Luis Alberto Lomel\'i}
\address{Luis Lomel\'i \\ Instituto de Matem\'aticas \\ PUCV \\ Valpara\'iso \\ Chile}
\email{luis.lomeli@pucv.cl}

\begin{abstract}
We prove three main results: all Langlands-Shahidi automorphic $L$-functions over function fields are rational; after twists by highly ramified characters our automorphic $L$-functions become polynomials; and, if $\pi$ is a globally generic cuspidal automorphic representation of a split classical group or a unitary group ${\bf G}_n$ and $\tau$ a cuspidal (unitary) automorphic representation of a general linear group, then $L(s,\pi \times \tau)$ is holomorphic for $\Re(s) > 1$ and has at most a simple pole at $s=1$. We also prove the holomorphy and non-vanishing of automorphic exterior square, symmetric square and Asai L-functions for $\Re(s) > 1$. Finally, we complete previous results on functoriality for the classical groups over function fields.
\end{abstract}

\maketitle

\section*{introduction}

The Langlands-Shahidi method provides a link between the theory of Eisenstein series and automorphic $L$-functions for globally generic representations. The method was first developed by Shahidi over number fields \cite{shahidi1990} and is now available in the case of function fields \cite{lsU}. We are guided by the analogy that exists in modern Number Theory between number fields and global function fields. Already present in Langlands' work \cite{langlandsES}, Eisenstein series over number fields have a meromorphic continuation. Over function fields, a result of Harder proves that Eisenstein series are rational \cite{harder1974}. Based on this, we prove that all Langlands-Shahidi automorphic $L$-functions over function fields are rational (Theorem~\ref{rationalL}).

The connection between Eisenstein series and the Langlands-Shahidi local coefficient is possible via the Whittaker model of a globally generic representation. Locally, at every unramified place, we have the Casselman-Shalika formula \cite{cs1980}. At ramified places, we find it useful to show in \S~\ref{Whittaker:local} that the local Whittaker functional of the corresponding parabolically induced representation is rational, in fact a Laurent polynomial (Theorem~\ref{Whittaker:Ind}); completing the proof of Theroem~\ref{rationalL}. We base ourselves on the proof of holomorphy in the case of principal series found in \cite{cs1980}; we found \cite{casselmanletter} very useful while working with questions on Whittaker functionals of $\mathfrak{p}$-adic groups.

In \S~\ref{intertwining} we collect several facts about $L$-functions and intertwining operators. We prove a local irreducibility result for induced representations and recall an assumption made by Kim concerning normalized intertwining operators.  We then, in \S~\ref{global:poly}, extend to function fields a result that Kim and Shahidi \cite{ks2002} proved for number fields, stating that Langlands-Shahidi $L$-functions become holomorphic after globally twisting by a highly ramified character. We follow Kim-Shahidi and refine this result in positive characteristic to show that our automorphic $L$-functions become Laurent polynomials (Proposition~\ref{twistLS}).

In \S~\ref{transfer} we recall the notion of the Kazhdan transfer between close local fields. We rely on the work of Ganapathy, who produces a Kazhdan transfer between Hecke algerbras using the Iwahori filtration \cite{Ga2015}. The Kazhdan philosophy allows us to transfer results for admissible representations that are known in characteristic zero to characteristic $p$, and viceversa. For example, we make note that Langlands-Shahidi local $L$-functions, $\gamma$-factors and root numbers are respected by the transfer for all split groups. Additionally, we prove in characteristic $p$, a result that Lapid, Mui\'c and Tadi\'c proved in characteristic zero concerning the generic unitary dual of classical groups \cite{LaMuTa2004}.

We then focus, starting in \S~\ref{classical:holo}, on ${\bf G}_n$ either a split classical group or a quasi-split unitary group. First, we look at the Siegel Levi subgroup $\bf M$ of ${\bf G}_n$ \cite{lomeliSiegel}. In this case, we show the holomorphy and non-vanishing of exterior square, symmetric square and Asai $L$-functions over function fields for $\Re(s) > 1$ via the Langlands-Shahidi method (Proposition~\ref{extsymAholo}). This is necessary as a preliminary induction step towards the holomorphy of $L$-functions for the classical groups and is studied ``ab initio" in \cite{lomeliSiegel} via the Langlands-Shahidi method in positive characteristic.

Furhter information on the location of the poles of these $L$-functions is known over number fields, and  it might be worhwhile to look into these questions over function fields. For example, the approach of Bump and Ginzburg over number fields \cite{bg1992}, continued by Takeda in the twisted case \cite{takeda2014}, gives precise information about the holomorphy of $L^S(s,\pi,{\rm Sym}^2)$, having possible poles only at $s = 1$. Using the Langlands-Shahidi method, we explore twisted exterior and symmetric square $L$-functions in \cite{gl2015}. There, we have a Kazhdan transfer from close local fields of characteristic zero to characteristic $p$.

Now, let $\pi$ be a globally generic cuspidal automorphic representation of a quasi-split classical group ${\bf G}_n$ and let $\tau$ a cuspidal (unitary) automorphic representation of ${\rm GL}_m$ (${\rm Res}\,{\rm GL}_m$ in the case ${\bf G}_n$ is a unitary group). Over function fields we can establish and refine a number fields result of \cite{kim2000, ks2004} in these cases. More precisely, we prove that $L(s,\pi \times \tau)$ is holomorphic for $\Re(s) > 1$ and has at most a simple pole at $s=1$ (Theorem~\ref{globholo}).

In \S~\ref{classical:fff}, we go back and show that our results on functoriality for the split classical groups of \cite{lomeli2009} are valid for any global function field $k$. Previously, we had made the assumption ${\rm char}(k) \neq 2$. In particular, Theorem~9.14 of [\emph{loc.\,cit.}] (Ramanujan Conjecture) and Theorem~4.4(iii) of \cite{ls} (Riemann Hypothesis) are now valid without this restriction.

\subsection*{Acknowledgments} I would like to thank G\"unter Harder for enlightening conversations that we held during the time the article was written. I thank Guy Henniart and Freydoon Shahidi for their encouragement to work on this project. I would like to thank W. Casselman for useful discussions concerning the rationality of the Langlands-Shahidi local coefficient. I also thank R. Ganapathy, V. Heiermann, D. Prasad and S. Varma for mathematical communications. I am indebted to the Mathematical Sciences Research Institute and the Max-Planck Institute f\"ur Mathematik for providing excellent working conditions while a Postdoctoral Fellow. I am grateful to the Instituto de Matem\'aticas PUCV, where I have found a great home institution. Work on this article was supported in part by MSRI under its NSF Grant, DMS 0932078 and Project VRIEA/PUCV 039.367/2016.

\section{Eisenstein series and $L$-functions over function fields}\label{Eisenstein:ff}

\subsection{Notation} Let $k$ be a global function field with field of constants $\mathbb{F}_q$ and ring of ad\`eles $\mathbb{A}_k$. At every place $v$ of $k$, we denote by $\mathcal{O}_v$ the corresponding ring of integers. Similarly, we let $\mathfrak{p}_v$ be the maximal prime ideal with uniformizer $\varpi_v$ and $q_v$ denotes the cardinality of $\mathcal{O}_v / \mathfrak{p}_v$.

Let $\bf G$ be a connected quasi-split reductive group defined over $k$. Fix a borel subgroup ${\bf B} = {\bf T}{\bf U}$ with maximal torus $\bf T$ and unipotent radical $\bf U$. Let $\Delta$ denote the simple roots. We consider only standard parabolic subgroups. There is a correspondence between subsets $\theta \subset \Delta$ and standard parabolic subgroup $\textbf{P}_\theta = \textbf{M}_\theta \textbf{N}_\theta$ of $\bf G$. Recall that if $\textbf{A}_\theta$ is $\left( \cap_{\alpha \in \theta} \ker(\alpha) \right)^0$, then $\textbf{M}_\theta$ is the centralizer of $\textbf{A}_\theta$ in $\bf G$ and $\textbf{N}_\theta$ is the unipotent radical of $\textbf{P}_\theta$. We often drop the subscripts  and write ${\bf P} = {\bf M}{\bf N}$. Given a place $v$ of $k$ and a reductive group $\bf H$ over $k$, we let $H_v$ denote the group of $k_v$-rational points.

Let $X^*({\bf M})_k$ be the set of $k$-rational characters of ${\bf M}$. For every place $v$ of $k$, we have the $k_v$-rational characters $X^*({\bf M})_v$. We let $X_{\rm nr}({\bf M})_k$ denote the set of unramified characters of ${\bf M}(k)$. And, similarly for $X_{\rm nr}({\bf M})_v$. Let
\begin{equation*}
   \mathfrak{a}_{{\bf M},\mathbb{C}}^* = X^*({\bf M})_k \otimes \mathbb{C}.
\end{equation*}
Let $\rho_{\bf P}$ denote half the sum of the positive roots of $\bf P$. The relation
\begin{equation*}
   q^{\left\langle s \otimes \chi, H_{\bf P}(\cdot)\right\rangle} = \left| \chi(\cdot) \right|_k^{s}
\end{equation*}
leads to a surjection
\begin{equation}\label{nrsurj}
   \mathfrak{a}_{M_v,\mathbb{C}}^* \twoheadrightarrow X_{\rm nr}({\bf M})_v.
\end{equation}
at every place $v$ of $k$. Fix a maximal compact open subgroup $\mathcal{K} = \prod_v \mathcal{K}_v$ of ${\bf G}(\mathbb{A}_k)$, as in \S~I.1.4 of \cite{mw1994}. Then
\begin{equation*}
   {\bf G}(\mathbb{A}_k) = {\bf P}(\mathbb{A}_k) \mathcal{K}.
\end{equation*}
For each $v$ of $k$, $H_{P_v}(\cdot)$ is trivial on $M_v \cap \mathcal{K}_v$. The character
\begin{equation*}
   q^{\left\langle s \otimes \chi_v, H_{P_v}(\cdot)\right\rangle} \in X_{\rm nr}({\bf M})_v,
\end{equation*}
extends to one of $G_v = {\bf G}(k_v)$.

A maximal parabolic $\bf P$ corresponds to a set of the form $\theta = \Delta - \left\{ \alpha \right\}$, i.e. $\textbf{P}=\textbf{P}_{\theta}$, for some simple root $\alpha$. In this case we fix a basis element of the space $\mathfrak{a}_{{\bf M},\mathbb{C}}^*$ by setting
\begin{equation}\label{basis}
   \tilde{\alpha} = \left\langle \rho_P, \alpha^\vee \right\rangle ^{-1} \rho_P,
\end{equation}
where $\alpha^\vee$ is the coroot corresponding to $\alpha$.

Let $\pi = \otimes' \pi_v$ be cuspidal (unitary) automorphic representation of ${\bf M}(\mathbb{A}_k)$. At every place $v$ of $k$, we have the induced representation
\begin{equation*}
   {\rm I}(s,\pi_v) = {\rm ind}_{P_v}^{G_v}(\pi_v \otimes q^{\left\langle s \tilde{\alpha} , H_{P_v}(\cdot) \right\rangle}).
\end{equation*}
Then, we globally have the restricted tensor product
\begin{equation*}
   {\rm I}(s,\pi) = \otimes' {\rm I}(s,\pi_v),
\end{equation*}
with respect to functions $f_{v,s}^\circ \in {\rm I}(s,\pi_v)$ that are fixed under the action of $\mathcal{K}_v$.

\subsection{Eisenstein series} Let $\phi : {\bf M}(k) \backslash {\bf M}(\mathbb{A}_k) \rightarrow \mathbb{C}$ be an automorphic form on the space of a cuspidal (unitary) automorphic representation $\pi$ of ${\bf M}(\mathbb{A}_k)$. Then $\phi$ extends to an automorphic function $\Phi : {\bf M}(k) {\bf U}(\mathbb{A}_k) \backslash {\bf G}(\mathbb{A}_k) \rightarrow \mathbb{C}$, as in \S~I.2.17 of \cite{mw1994}. For every $s \in \mathbb{C}$, set
\begin{equation*}
   \Phi_s = \Phi \cdot q^{\left\langle s \tilde{\alpha} + \rho_{\bf P}, H_{\bf P}(\cdot) \right\rangle}.
\end{equation*}
The function $\Phi_s$ is a member of the globally induced representation ${\rm I}(s,\pi)$. We use the notation of \S~1.4 and Remark~1.2 of \cite{lsU}, where $w_0=w_lw_{l,{\bf M}}$ (see also \S~2.1 below for more details on the local notation). We have the global intertwining operator
\begin{equation*}
   {\rm M}(s,\pi,\tilde{w}_0) : {\rm I}(s,\pi) \rightarrow {\rm I}(s',\pi'),
\end{equation*}
defined by
\begin{equation*}
   {\rm M}(s,\pi,\tilde{w}_0)f(g) = \int_{N'} f(\tilde{w}_0^{-1}ng) dn,
\end{equation*}
for $f \in {\rm I}(s,\pi)$. It decomposes into a product of local intertwining operators
\begin{equation*}
   {\rm M}(s,\pi,\tilde{w}_0) = \otimes_v \, {\rm A}(s,\pi_v,\tilde{w}_0),
\end{equation*}
which are precisely those appearing in the definition of the Langlands-Shahidi local coefficient (see \S~1.5 of \cite{lsU}).

The following crucial result is found in \cite{harder1974} for split groups under certain assumptions, and is generalized in \S~IV of \cite{mw1994} to the case at hand.

\begin{theorem}[Harder] \label{ESrationality}
The Eisenstein series
\begin{equation*}
   E(s,\Phi,g,{\bf P}) = \sum_{\gamma \in {\bf P}(k) \backslash {\bf G}(k)} \Phi_s(\gamma g)
\end{equation*}
converges absolutely for $\Re(s) \gg 0$ and has a meromorphic continuation to a rational function on $q^{-s}$. Furthermore
\begin{equation*}
   {\rm M}(s,\pi) = \otimes_v \, {\rm A}(s,\pi_v,\tilde{w}_0)
\end{equation*}
is a rational operator in the variable $q^{-s}$.
\end{theorem}

\subsection{$L$-groups} Let $\Gamma$ be the Galois group over the separable closure $k_s$ of $k$. The fixed Borel corresponds to the simple roots $\Delta$. This determines a based root datum $\Psi_0 = (X^*,\Delta, X_*,\Delta^\vee)$. The dual root datum $\Psi_0^\vee = (X_*,\Delta^\vee,X^*,\Delta)$ determines the Chevalley group ${}^LG^\circ$ over $\mathbb{C}$. The $L$-group of ${\bf G}$ is the semidirect product
\begin{equation*}
   {}^LG = {}^LG^\circ \rtimes \Gamma.
\end{equation*}
The based root datum $\Psi^\vee$ fixes a Borel subgroup ${}^LB$, and we have all standard parabolic subgroups of the form ${}^LP = {}^LP^\circ \rtimes \Gamma$. The Levi subgroup of ${}^LP$ is of the form ${}^LM = {}^LM^\circ \rtimes \Gamma$, while the unipotent radical is given by ${}^LN = {}^LN^\circ$. We refer to \cite{borel} for more details on the semi-direct structure of ${}^LG$.

\subsection{Langlands-Shahidi $L$-functions}\label{notation:LS}
Fix a pair of quasi-split reductive group schemes $({\bf G},{\bf M})$ such that ${\bf M}$ is a Levi component of a parabolic subgroup ${\bf P} = {\bf M}{\bf N}$ of $\bf G$. Let $r: {}^LM \rightarrow {\rm End}({}^L\mathfrak{n})$ be the adjoint representation of ${}^LM$ on the Lie algebra ${}^L\mathfrak{n}$ of ${}^LN$. It decomposes into irreducible constituents
\begin{equation*}
   r = \oplus_{i=1}^{m_r} r_i.
\end{equation*}
We study $L$-functions that arise from such $r_i$. We refer to \cite{lsU} for more details.

Locally, we let $\mathscr{L}_{\rm loc}(p)$ denote the category of triples $(F,\pi,\psi)$ consisting of: $F$ a non-archimedean local field of characteristic $p$; $\pi$ a generic representation of ${\bf M}(F)$; and $\psi : F \rightarrow \mathbb{C}^\times$ a continuous non-trivial character. We say $(F,\pi,\psi) \in \mathscr{L}_{\rm loc}(p)$ is unramified, if $\pi$ has an Iwahori fixed vector.

Globally, let $\mathscr{L}_{\rm glob}(p)$ be the category of quadruples $(k,\pi,\psi,S)$ consisting of: $k$ a global function field with field of constants $\mathbb{F}_q$; 
$\pi = \otimes' \pi_v$ a globally generic cuspidal automorphic representation of ${\bf M}(\mathbb{A}_k)$; $\psi : k \backslash \mathbb{A}_k \rightarrow \mathbb{C}^\times$ a non-trivial character; and, $S$ a finite set of places of $k$ such that $\pi_v$ and $\psi_v$ are unramified for $v \notin S$.

Theorem~6.1 of \cite{lsU} associates to each triple $(F,\pi,\psi) \in \mathscr{L}_{\rm loc}(p)$ the local factors:
\begin{equation*}
   \gamma(s,\pi,r_i,\psi), \ L(s,\pi,r_i), \text{ and } \varepsilon(s,\pi,r_i,\psi).
\end{equation*}
Globally, we have partial $L$-functions
\begin{equation*}
   L^S(s,\pi,r_i) = \prod_{v \notin S} L(s,\pi_v,r_{i,v}),
\end{equation*}
for $(k,\pi,\psi,S) \in \mathscr{L}_{\rm glob}(p)$. And, finally, completed global $L$-functions and $\varepsilon$-factors:
\begin{equation*}
   L(s,\pi,r_i) = \prod_{v} L(s,\pi_v,r_{i,v}) \text{ and } \varepsilon(s,\pi,r_i) = \prod_v \varepsilon(s,\pi_v,r_{i,v},\psi_v).
\end{equation*}

\subsection{Globally generic representations and rationality}

In the Langlands-Shahidi method, the quadruples we consider $(k,\pi,\psi,S) \in \mathscr{L}_{\rm glob}(p)$ require that $\pi = \otimes' \pi_v $ be globally generic. In particular, we can assume that $\pi$ is globally $\psi$-generic. By definition, there is a cusp form $\varphi$ in the space of $\pi$ such that
\begin{equation*}
   W_{M,\varphi}(m) = \int_{{\bf U}_M(K) \backslash {\bf U}_M(\mathbb{A}_k)} \varphi(um) \overline{\psi}(u) \, du \neq 0.
\end{equation*}
In this case, the Fourier coefficient of the Eisenstein series $E(s,\Phi,g)$ is given by
\begin{equation*}
   E_\psi(s,\Phi,g,{\bf P}) = \int_{{\bf U}(K) \backslash {\bf U}(\mathbb{A}_k)} E(s,\Phi,ug) \overline{\psi}(u) \, du.
\end{equation*}
The Fourier coefficients of Eisenstein series are also rational functions on $q^{-s}$ \cite{harder1974}. We use this to prove the following result.

\begin{theorem}\label{rationalL}
Let $(k,\pi,\psi,S) \in \mathscr{L}_{\rm glob}(p)$. Then each $L$-function $L(s,\pi,r_i)$ converges absolutely for $\Re(s) \gg 0$ and has a meromorphic continuation to a rational function in $q^{-s}$.
\end{theorem}
\begin{proof}
The generic assumption on $\pi$, makes it possible to further decompose the Fourier coefficients of Eisenstein. From \S~5 of \cite{lsU}, we have the connection between the global Whittaker model and the local ones:
\begin{equation*}
   E_\psi(s,\Phi,g,{\bf P}) = \prod_v \lambda_{\psi_v}(s,\pi_v)({\rm I}(s,\pi_v)(g_v)f_{s,v}),
\end{equation*}
with $f_s \in {\rm I}(s,\pi)$, $f_{s,v} = f_{s,v}^\circ$ for all $v \notin S$. 
Furthermore, we have Corollary~5.2 of [\emph{loc.\,cit.}]:
\begin{equation}\label{eq1rationalL}
   E_\psi(s,\Phi,g,{\bf P}) = \prod_{v \in S} \lambda_{\psi_v}(s,\pi_v)({\rm I}(s,\pi_v)(g_v)f_{s,v}) \prod_{i=1}^{m_r} L^S(1+is,\pi,r_i)^{-1}.
\end{equation}
We explore the details of the local Whittaker functionals $\lambda_{\psi_v}(s,\pi_v)({\rm I}(s,\pi_v)(g_v)f_{s,v})$ in the next section. In particular, Theorem~\ref{precise:Wthm} proves that the local Whittaker functionals are polynomials in $\left\{ q_v^{s},q_v^{-s} \right\}$, where $q_v = q^{\deg v}$. Now, the Fourier coefficients $E_\psi(s,\Phi,g,{\bf P})$ are rational on $q^{-s}$. Also, recall that each $L^S(s,\pi,r_i)$ are absolutely convergent for $\Re(s) \gg 0$, Theorem~13.2 of \cite{borel}. Then, we can conclude that the product
\begin{equation*}
   \prod_{i=1}^{m_r} L^S(1+is,\pi,r_i)
\end{equation*}
extends to a rational function on $q^{-s}$. The induction step found in \S~6 of \cite{lsU} allows us to isolate each $L$-function to conclude that each
\begin{equation*}
   L^S(1+is,\pi,r_i)
\end{equation*}
is rational in the variable $q^{-s}$. Hence, each $L^S(s,\pi,r_i)$ is rational. Furthermore, Theorem~6.1 of [\emph{loc.\,cit.}], gives that locally each $L$-function
\begin{equation*}
   L(s,\pi_v,r_{i,v}), \text{ for } (k_v,\pi_v,\psi_v) \in \mathscr{L}_{\rm loc}(p),
\end{equation*}
is a rational function on $q_v^{-s}$. Hence, the completed $L$-function
\begin{equation*}
   L(s,\pi,r_i) = \prod_{v \in S} L(s,\pi_v,r_{i,v}) L^S(s,\pi,r_i)
\end{equation*}
meromorphically continues to a rational function in the variable $q^{-s}$.
\end{proof}

\section{On the rationality of local Whittaker models}\label{Whittaker:local}

This section is local in nature and we work over any non-archimedean local field $F$ with residue field $\mathbb{F}_q$. We take the opportunity to make available the details of the rationality of local Whittaker functionals (Theorem~\ref{Whittaker:Ind}, and more precisely Theorem~\ref{precise:Wthm}).

\subsection{Induced representations and Whittaker models}\label{Ind:notation} We fix a pair of quasi-split reductive groups $({\bf G},{\bf M})$ and use the local notation of \S~\ref{notation:LS}. Let $(F,\pi,\psi) \in \mathscr{L}_{\rm loc}(p)$ and denote the space of $\pi$ by $V$. For $\nu \in \mathfrak{a}_{M,\mathbb{C}}^*$, let
\begin{equation*}
   {\rm I}(\nu,\pi) = {\rm ind}_P^G(\pi \otimes q^{\left\langle \nu,H_P(\cdot)\right\rangle}),
\end{equation*}
where ${\rm ind}$ denotes normalized unitary parabolic induction, and we denote the space of ${\rm I}(\nu,\pi)$ by ${\rm V}(\nu,\pi)$. We will often abuse notation and identify the representation ${\rm I}(\nu,\pi)$ with its space ${\rm V}(\nu,\pi)$. Furhtermore, when the parabolic subgroup $\bf P$ of $\bf G$ is clear from context, we simply write ${\rm Ind}(\pi)$ instead of ${\rm I}(0,\pi)$ or ${\rm ind}_P^G(\pi)$.

Equation~\eqref{nrsurj}, is locally given as follows:
\begin{equation*}
   \mathfrak{a}_{M,\mathbb{C}}^* \twoheadrightarrow X_{\rm nr}(\textbf{M})
\end{equation*}
associates the $F$-rational character $q^\nu = q^{\left\langle \nu, H_\theta(\cdot) \right\rangle}$ to the element $\nu \in \mathfrak{a}_{M,\mathbb{C}}^*$. The kernel of this map is of the form $\frac{2\pi i}{\log q}\Lambda$, for a certain lattice $\Lambda$ of $\mathfrak{a}_{M,\mathbb{C}}^*$. This surjection gives $X_{\rm nr}(\textbf{M})$ the structure of a complex algebraic variety of dimension $d=\dim_\mathbb{R} \mathfrak{a}_M$. Thus, there are notions of polynomial and rational functions on $X_{\rm nr}({\bf M})$ (see \S~4 of \cite{Wa2003}). 

For an irreducible admissible representation, the space of Whittaker functionals has dimension at most one \cite{shalika1974}. In this article we assume $(F,\pi,\psi) \in \mathscr{L}_{\rm loc}(p)$ has $\pi$ generic and therefore there is a unique Whittaker functional on $V$, up to multiplication by a constant. Let us state a preliminary version of the main result of this section; it will be stated more precisely as Theorem~\ref{precise:Wthm}.

\begin{theorem}\label{Whittaker:Ind}
Let $(F,\pi,\psi) \in \mathscr{L}_{\rm loc}(p)$ and let $\nu \in \mathfrak{a}_{M,\mathbb{C}}^*$. Then the induced representation ${\rm I}(\nu,\pi)$ is also generic. It has a Whittaker functional $\lambda_\psi(\nu,\pi)$ which is a polynomial on $\left\{ q^s, q^{-s} \right\}$. 
\end{theorem}

Given $(F,\pi,\psi) \in \mathscr{L}_{\rm loc}(p)$, we have Whittaker functionals for the representations $(\pi,V)$ and $\left( {\rm I}(\nu,\pi),{\rm V}(\nu,\pi) \right)$. Care must be taken when going from one generic character to another. Because of this, we now revisit local Whittaker models.

\subsection{Local Whittaker models} We write algebraic group schemes defined over $F$ in bold capital letters, e.g., $\bf G$, and the corresponding group of rational points in regular capital letters, e.g. $G = {\bf G}(F)$. The roots of $\bf G$ with respect to a fixed maximal split torus $\textbf{T}_0$ are denoted by $\Sigma$, the positive roots by $\Sigma^+$, which depend on the choice of Borel $\bf B$. We denote the simple roots by $\Delta$. We write $\Sigma_r$ and $\Sigma_r^+$ to be the corresponding sets of reduced roots.

For each $\alpha \in \Sigma$, let $\textbf{N}_\alpha$ be the subgroup of $\bf U$ whose Lie algebra is $\mathfrak{g}_\alpha + \mathfrak{g}_{2\alpha}$. To define a non-degenerate character of $U$, first write ${\bf U} = \prod_{\alpha \in \Sigma_r^+} {\bf N}_\alpha$. The group $\prod_{\alpha \in \Sigma^+ - \Delta} {\bf N}_\alpha$ is a normal subgroup of ${\bf U}$ and 
\begin{equation}
   {\bf U} \ / \prod_{\alpha \in \Sigma^+ - \Delta} {\bf N}_\alpha \cong \prod_{\alpha \in \Delta} {\bf N}_\alpha \slash {\bf N}_{2\alpha}.
\label{uwhittaker}
\end{equation}
If $\psi_\alpha:N_\alpha / N_{2\alpha} \rightarrow \mathbb{C}$ is a smooth character for each $\alpha \in \Delta$, then $\psi=\prod_{\alpha \in \Delta}{\psi_\alpha}$ determines a character of $U$ via the projection map ${\bf U} \rightarrow {\bf U} \ / \prod_{\alpha \in \Sigma^+ - \Delta} {\bf N}_\alpha$ and the isomorphism of equation~\eqref{uwhittaker}. A character $\psi : U \rightarrow \mathbb{C}$ defined this way is called \emph{non-degenerate} if each $\psi_\alpha$ is non-trivial.

Fix $\psi$, a non-degenerate character of $U$. Recall that an admissible representation $(\rho,W)$ of $G$ is said to be $\psi$-\emph{generic} or \emph{non-degenerate} if there exists a linear functional $\lambda_\psi: W \rightarrow \mathbb{C}$ such that $\lambda_\psi(\rho(u)w)=\psi(u)\lambda_\psi(w)$ for all $u \in U$; $\lambda_\psi$ is called a \emph{Whittaker functional}.

Having fixed $\psi$, one can define a non-degenerate character $\psi_{\tilde{w}_0}$ of the unipotent radical $M \cap U$ of $M$ so that $\psi$ and $\psi_{\tilde{w}_0}$ are $\tilde{w}_0$-\emph{compatible}, that is to say
\begin{equation}\label{compatible}
   \psi_{\tilde{w}_0}(u)=\psi(\tilde{w}_{0} u \tilde{w}_{0}^{-1}), \ u \in M \cap U.
\end{equation}
Here, $\tilde{w}$ denotes a representative in $N/T_0$ of a Weyl group element $w$. The question of varying Weyl group element representatives, as well as Haar measures, is taken up in \S~2 of \cite{lsU}. The results of this section are valid for any fixed set of representatives and any two fixed non-degenerate characters $\psi$ of $U$ and $\psi_{\tilde{w}_0}$ of $U_M$ which are compatible in the sense of \eqref{compatible}.

\subsection{Basic facts about twisted Jacquet modules}

The twisted Jacquet module of any smooth representation $(\sigma,V)$ of $U$, is defined to be
\begin{equation*}
   V_{\psi,U}=V/V_\psi(U),
\end{equation*}
where $V_\psi(U)$ is the span of $\{ \sigma(u)v-\psi(u)v \vert u \in U, v\in V \}$. Another characterization of $V_\psi(U)$ is as the set of those $v \in V$ such that there is a compact open subgroup $U_0 \subset U$ with the property that
\begin{equation*}
   \int_{U_0} \psi^{-1}(u)\sigma(u)vdu=0.
\end{equation*}
The usual Jacquet module $V_U$ is equal to $V_{1,U}$ in this notation. From \cite{cs1980} we have the following proposition:

\begin{proposition}\label{twistedprop}
Twisted Jacquet modules satisfy the following properties:
\begin{itemize}
   \item[(i)] The functor $V\rightarrow V_{\psi,U}$ is exact.
   \item[(ii)] If  $V^\prime$ is any space on which $U$ acts by $\psi$ then the map $V \rightarrow V_{\psi,U}$ induces an isomorphism
   \begin{equation*}
	{\rm Hom}_U(V,V^\prime) \cong {\rm Hom}_{\mathbb{C}}(V_{\psi,U},V^\prime).
   \end{equation*}
   \item[(iii)] Let $\Omega$ be the $U$-morphism ${\rm I}_U^G(\mathbb{C}_\psi) \rightarrow \mathbb{C}_\psi, f \mapsto f(1)$. Then for any smooth character $\psi$ of $U$, and $V$ a smooth representation of $G$. Composition with $\Omega$ induces an isomorphism
   \begin{equation*}
      {\rm Hom}_G(V,{\rm I}_U^G(\mathbb{C}_\psi)) \cong {\rm Hom}_{\mathbb{C}}(V_{\psi,U}, \mathbb{C}).
   \end{equation*}
   \item[(iv)] If $U_1$ and $U_2$ are two subgroups of $U$, with $U=U_1U_2$, then
   \begin{equation*}
	V_\psi(U)=V_\psi(U_1)+V_\psi(U_2).
   \end{equation*}
\end{itemize}
\end{proposition}

\subsection{The Whittaker model of an induced representation}

Fix a parabolic subgroup ${\bf P} = {\bf P}_\theta$ with Levi $\bf M$. Also, fix non-degenerate characters $\psi$ and $\psi_{\tilde{w}_0}$ which are compatible in the sense of equation~\eqref{compatible}. Let $(F,\pi,\psi) \in \mathscr{L}_{\rm loc}(p)$ be such that $\pi$ is a $\psi_{\tilde{w}_0}$-generic representation of $M$ and denote its Whittaker functional by $\lambda_{\psi_{\tilde{w}_0}}$. Let $\otheta=\tilde{w}_0\theta$ be the conjugate of $\theta$ in $\Delta$. We then have the standard parabolic subgroup $\overline{\bf P} = {\bf P}_{\overline{\theta}}$ of $\bf G$, and we write $\overline{\bf M} = {\bf M}_{\overline{\theta}}$ for its Levi and $\overline{\bf N} = {\bf N}_{\overline{\theta}}$ for its unipotent radical.

Let $W$ be the Weyl group of $\Sigma$, and for each $\alpha \in \Delta$ let $w_\alpha$ be its corresponding reflection. We let $W_{\bf M}$ be the subgroup of $W$ generated by $w_\alpha, \ \alpha \in \theta$. Denote by $w_l$ and $w_{l,{\bf M}}$ the longest elements of $W$ and $W_{\bf M}$, respectively. Let $[W_{\bf M} \backslash W] = \left\{w \in W \vert w^{-1}\theta > 0 \right\}$, then $w_0=w_lw_{l,\bf M}$ will be the longest element of $[W_{\bf M} \backslash W]$. Weyl group element representatives $\tilde{w}$ are chosen as in \cite{lsU}.

We construct the Whittaker functional $\lambda_\psi(\nu,\pi): {\rm V}(\nu,\pi) \rightarrow \mathbb{C}$ of ${\rm I}(\nu,\pi)$ from $\lambda_{\psi_{\tilde{w}_0}}: V \rightarrow \mathbb{C}$. We first define it on a subspace $I_{d_P}$ of $I=I(\nu,\pi)$, to which we now turn. Recall that the Bruhat decomposition 
\begin{equation}\label{bruhat}
   G = \coprod_{w \in [W_{\bf M} \backslash W]} P w B
\end{equation}
gives rise to the Bruhat order among the double cosets $P w B$ of $G$. This partial order is given by $P w_1 B \prec P w_2 B $, whenever $P w_1 B$ is contained in the closure of $P w_2 B$. Here, $P w_0^{-1} B$ is maximal and is the unique open double coset which is dense in $G$. Write $d(w)= \dim(P \backslash P w B)$ and in particular $d_P = d(w_0)= \dim(P \backslash P w_0 B)$. We need a result from Casselman's notes (section~6 of \cite{casselman}), which states that there is a filtration of $I$ by $B$-stable subspaces
\begin{equation*}
	0\subset I_{d_P} \subset \cdots \subset I_0 \subset I,
\label{filtration}
\end{equation*}
where
\begin{equation*}
	I_{n} = \{ f \in I \ \vert \ {\rm supp}(f) \subset \bigcup_{d(w) \geq n} P w B \}.
\end{equation*}
Moreover
\begin{equation*}	
   I_n / I_{n+1} = \bigoplus_{d(w)=n} I_w, 
\end{equation*}
\begin{equation*}	
   I_w = \text{c-}{\rm Ind}_{wPw^{-1} \cap B }^{B} (w(\pi \otimes q^{\left\langle
   \nu,H_P(\cdot)\right\rangle} \delta_P^{1/2})). 
\end{equation*}
The subspace of importance to us is 
\begin{equation}\label{Idtheta}
   I_{d_P} = \left\{ f \in I(\nu,\pi) \ \vert \ {\rm supp}(f) \subset P w_0 B   \right\},
\end{equation}
and will also be denoted by $I(\nu,\pi)_{d_P}$, since we may want to vary $\nu$.

For functions $f$ in $I(\nu,\pi)_{d_P}$, let
\begin{equation*}
   \lambda_\psi(\nu,\pi)f 
   = \int_{\overline{N}} \lambda_{\psi_{w_0}}(f(w_0^{-1}n))\opsi(n)dn.
\end{equation*}

\begin{proposition}\label{welldefined}
For each $\nu \in \mathfrak{a}_{M,\mathbb{C}}^*$, $\lambda_\psi(\nu,\pi)$ is a well defined function on $I(\nu,\pi)_{d_P}$.
\end{proposition}

\begin{proof}
There is an isomorphism 
\begin{equation*}
   {\rm I}(\nu,\pi)_{d_P} \cong \text{c-}{\rm Ind}_{w_0 P w_{0}^{-1} \cap B}^{B} w_0(\pi \otimes q^{\left\langle \nu,H_P(\cdot)\right\rangle} \delta_P^{1/2}),
\end{equation*}
where $w_0(\pi \otimes q^{\left\langle \nu,H_P(\cdot)\right\rangle} \delta_P^{1/2})(p) = \pi \otimes q^{\left\langle \nu,H_P(\cdot)\right\rangle} \delta_P^{1/2}(w_0^{-1} p w_0), \ \forall p \in w_0 P w_0^{-1}$. The isomorphism takes a function $ f \in I(\nu,\pi)_{d_P}$ to the function $x \mapsto f(w_0^{-1}x)$, see Proposition~6.3.2 of \cite{casselman}. 

Since c-${\rm Ind}$ is the non-normalized induced representation consisting of functions of compact support modulo $w_0 P w_0^{-1} \cap B$, the function $x \mapsto f(w_0^{-1}x)$ restricted to $U$ has compact support modulo $w_0 P w_0^{-1} \cap U = \overline{M} \cap U$ (notice that $w_0^{-1} \alpha < 0, \forall \alpha \in \Sigma^+ - \Sigma_{\theta}^+$) and since $U = \overline{N} \left( \overline{M} \cap U \right)$, the function $x \mapsto f(w_0^{-1}x)$ as a function on $\overline{N} = w_0 P w_0^{-1} \backslash U$ has compact support.
\end{proof}

Let $G^*$ be the complement of $P w_0 B$ in $G$, which is a closed subset of $G$. Define $J=J(\nu,\pi)$ to be the space of locally constant functions $f:G^* \rightarrow V$ such that $f(pg)=\pi(m_p)q^{\left\langle \nu+\rho_P, H_P(m_p) \right\rangle}f(g)$, for all $p=m_pn_p \in P$, $g \in G^*$. Restriction gives a morphism of $U$-spaces from $I$ to $J$. There is then an exact sequence
\begin{equation}
   0 \rightarrow I_{d_P} \rightarrow I \rightarrow J \rightarrow 0.
\label{Usequence}
\end{equation}
We have the following:

\begin{lemma}\label{twistedj}
   $J_{\psi,U}=0$.
\end{lemma}
\begin{proof}
Because of the fact that $J=\oplus_{w \neq w_0} I_w$ and the exactness property of Jacquet modules, Proposition~\ref{twistedprop}, it is enough to show that $(I_w)_{\psi,U}=0$ for $w \in [W_{\bf M} \backslash W]$, $w \neq w_0$. This, in turn, is equivalent to showing that
\begin{equation*}
	{\rm Hom}_N(I_w,\mathbb{C}_\psi) \cong {\rm Hom}_{\mathbb{C}}((I_w)_{\psi,U},\mathbb{C})=0.
\end{equation*}
An element $\Phi \in {\rm Hom}_N(I_w,\mathbb{C}_\psi)$ is in the dual of $I_w$ and is an eigenvector for $U$ with eigencharacter $\opsi$, thus $\Phi \in \widetilde{I}_w$ (the contragredient representation of $I_w$).

As a $U$-space 
\begin{equation*}
	I_w=\text{c-}{\rm Ind}_{wUw^{-1}\cap U}^{U}(w(\pi \otimes q^{\left\langle \nu,H_P(\cdot)\right\rangle})), 
\end{equation*}
and \cite{casselman} 2.4.2 says that
\begin{equation*}
	\widetilde{I_w}\cong {\rm I}_{wUw^{-1}\cap U}^{U}(w(\widetilde{\sigma})), 
	\ \sigma = \pi \otimes q^{\left\langle \nu,H_P(\cdot)\right\rangle}.
\end{equation*}
where $\sigma = \pi \otimes q^{\left\langle \nu,H_P(\cdot)\right\rangle}$. Suppose $F:U \rightarrow \widetilde{V}$ is the function corresponding to $\Phi$, then on one hand
\begin{equation*}
	F(u_0u)=\widetilde{\sigma}(\tilde{w}^{-1}u_0\tilde{w})F(u), \ \forall u_0 \in wUw^{-1} \cap U, u \in U,
\end{equation*} 
and on the other it is an eigenvector for $U$ so that
\begin{equation*}
   F(u)=\opsi(u)F(1), \ \forall u\in U,
\end{equation*}
thus
\begin{equation*}
   \widetilde{\sigma}(\tilde{w}^{-1}u_0\tilde{w})F(1)=\opsi(u_0)F(1), \ \forall u_0\in wUw^{-1} \cap U. 
\end{equation*}
The representation $\widetilde{\sigma}$ is trivial on $N$, so in order to get $F=0$ all that is needed is to find a $u_0\in wN w^{-1} \cap U$ for which $\psi(u_0)\neq 1$ . But for $w\in [W_{\bf M} \backslash W], \ w\neq w_0$, there exists an $\alpha \in \Delta$ with $w\alpha \in \Sigma^+ - \Sigma_{\bf M}^+$. For such an $\alpha$, choose $u_0\in N_\alpha$ with $\psi(u_0)\neq 1$, which can be done since $\psi$ is non-degenerate.
\end{proof}

Given $p \in P$, let $m_p$ and $n_p$ be the components of $p$ when using the Levi decomposition of $P$, so that $p=m_pn_p \in M N$. We also fix a left Haar measure on $P$. The following is Proposition~1.1 of \cite{lomeli2009} and its corollary.

\begin{proposition}
For each $\nu \in \mathfrak{a}_{M,\mathbb{C}}^*$, there is a surjective map $\mathcal{P}_\nu :C_c^{\infty}(G,V) \rightarrow I(\nu,\pi)$, given by
\begin{equation*}
   \mathcal{P}_\nu \varphi(g) = \int_{P} q^{- \left\langle \nu,H_P(m_p)\right\rangle}
   \delta_P^{1/2}(p) \pi^{-1}(m_p) \varphi(pg)dp \ .
\label{Pnu}
\end{equation*}
\end{proposition}

\begin{corollary}\label{Pnucor}
For each $\nu \in \mathfrak{a}_{M,\mathbb{C}}^*$:
\begin{enumerate}
\item Given a compact open subgroup $K$ of $G$ and a function $f \in I(\nu,\pi)^K$, there is a (right) $K$-invariant function $\varphi \in C_c^\infty(G,V)$ with $\mathcal{P}_\nu \varphi = f$;
\item $\mathcal{P}_\nu$ sends $C_c^{\infty}(P w_0 B,V)$ onto $I(\nu,\pi)_{d_P}$. 
\end{enumerate}
\end{corollary}

We now turn towards the rationality of local Whittaker models. In fact, they are Laurent polynomials. We begin by proving this for ${\rm I}_{d_P}$. We use the notion of a function on $q^\nu \in X_{\rm nr}(\textbf{M})$. For this, we fix a basis $\chi_1, \ldots, \chi_d$ of the free $\mathbb{Z}$-module $X(\textbf{M})$, so that $\nu = s_1\otimes \chi_1 + \cdots + s_d\otimes \chi_d$, $s_i \in \mathbb{C}$. Then, what is meant by a polynomial (resp. rational) function on $q^\nu$ is simply a polynomial (resp. rational) function on the variables $Z_1,Z_1^{-1}, \ldots, Z_d,Z_d^{-1}$, where $Z_i = q^{s_i}$. Note that the affine algebra of $\mathbb{C}^\times$ can be thought as $\mathbb{C} \left[ Z,Z^{-1} \right]$. When $\bf M$ is maximal, the basis is dictated by $\tilde{\alpha}$ of equation~\eqref{basis}.

\begin{lemma}
   For any $\varphi \in C_c^{\infty}(P w_0 B,V)$, the function $\lambda_\psi(\nu,\pi) \mathcal{P}_\nu \varphi$ is a holomorphic function on $\nu \in \mathfrak{a}_{M,\mathbb{C}}^*$. In fact, $\lambda_\psi(\nu,\pi) \mathcal{P}_\nu \varphi$ is a polynomial in $q^\nu$.
\label{holomorphicl}
\end{lemma}

\begin{proof}
If $\varphi$ is such a function, then $\mathcal{P}_\nu \varphi \in I_{d_P}$ and
\begin{align*}
   \lambda_\psi(&\nu,\pi) \mathcal{P}_\nu \varphi
	= \int_{\overline{N}}\lambda_{\psi_{w_0}}(\mathcal{P}_\nu \varphi(w_0^{-1}n))\opsi(n)dn \\
	&= \lambda_{\psi_{w_0}} \left( \int_{\overline{N}} \int_{P} 
	q^{- \left\langle \nu,H_P(m_p)\right\rangle} \delta_P^{1/2}(p) \pi^{-1}(m_p)
	\varphi(pw_0^{-1}n) \opsi(n) \ dp \ dn \right).
\end{align*}
For $y=pw_0^{-1}n \in P w_0^{-1} \overline{N}$ let 
\begin{equation*}
	\varPhi(y) = \varPhi (pw_0^{-1}n) = \pi^{-1}(m_p) \delta_P^{1/2}(p) f(pw_0^{-1}n) \opsi(n)\ ,
\end{equation*}
and extend $q^{\left\langle v,H_P(\cdot) \right\rangle}$ to $P w_0^{-1}\overline{N}$ in an obvious way, so that 
\begin{equation*}
	\lambda_\psi(\nu,\pi)f = \lambda_{\psi_{w_0}} \left( \int_{P w_0^{-1} \overline{N}} q^{- \left\langle \nu, H_P(y) \right\rangle} \varPhi(y)dy \right).
\end{equation*}
Because $\varphi$ is a locally constant function of compact support, then so is $\varPhi$. Hence this last integral, as with all of the above integrals, can be written as a finite sum. By doing so, it is not hard to see that $\lambda_\psi(\nu,\pi)\mathcal{P}_\nu \varphi$ is a polynomial in $q^\nu$.
\end{proof}

\begin{remark}
Assuming Proposition~\ref{welldefined}, the following proposition is derived in \S~8.2 \cite{kimnotes2004}. However, no proof of holomorphicity of the Whittaker model is found in the literature other than the pricipal series case of \cite{cs1980}. For completion, we extend these results here to the general case.
\end{remark}

\begin{proposition}\label{wproperty}
For any  $f \in I(\nu,\pi)_{d_P}$,
\begin{equation*}
   \lambda_\psi(\nu,\pi)(I(\nu,\pi)(u)\cdot f)= \psi(u) \lambda_\psi(\nu,\pi)f \ .
\end{equation*}
\end{proposition}

For any compact open subgroup $U_1$ of $U$, there is a projection operator $\mathcal{P}_{\psi,U_1} : {\rm I}(\nu,\pi) \rightarrow {\rm I}(\nu,\pi)_{\psi,U_1}$ , given by
\begin{equation}
   \mathcal{P}_{\psi,U_1}f(g) = ({\rm meas} U_1)^{-1} \int_{U_1} f(gu) \opsi(u)du \ .
\end{equation}

\begin{lemma}
   For any $f\in {\rm I}(\nu,\pi)_{d_P}$ and $U_1$ a compact open subgroup of $U$,
   \begin{equation*}
      \lambda_\psi(\nu,\pi)(\mathcal{P}_{\psi,U_1} f) = \lambda_\psi(\nu,\pi)f
   \end{equation*}
\end{lemma}
\begin{proof}
   Notice that
   \begin{align*}
      \lambda_\psi&(\nu,\pi)(\mathcal{P}_{\psi,U_1} f) \\
      &= ({\rm meas} U_1)^{-1} \int_{\overline{N}} \int_{U_1} \lambda_{\psi_{w_0}}(f(w_0^{-1}nu))
         \opsi(u)\opsi(n) \, du dn \\
      &= ({\rm meas} U_1)^{-1} \int_{U_1} \lambda_{\psi}(I(\nu,\pi)(u) \cdot f(w_0^{-1}n))
         \opsi(u) \, du \ ,
        \end{align*}
then Proposition~\ref{wproperty} yields the required identity.
\end{proof}

The next remark follows from the definition, Proposition~\ref{welldefined} and the computations for Proposition~\ref{wproperty}. It points towards the steps needed to extend the definition of $\lambda_\psi(\nu,\pi)$ from ${\rm I}(\nu,\pi)_{d_P}$ to all of ${\rm I}(\nu,\pi)$. 

\begin{remark}
Let $f \in {\rm I}(\nu,\pi)$. Suppose we have a compact open subgroup $U_1$ of $U$, such that $\mathcal{P}_{\psi,U_1}f \in I(\nu,\pi)_{d_P}$. Then for any compact open subgroup $N_1$ of $\overline{N}$ with $N_1 \supset U_1 \cap \overline{N}$
\begin{equation*}
   \int_{N_1} \lambda_{\psi_{w_0}} (f(w_0^{-1}n))\opsi(n)  dn 
   = \int_{N_1} \lambda_{\psi_{w_0}} (\mathcal{P}_{\psi,U_1} f(w_0^{-1}n)) \opsi(n)  dn\ .
\end{equation*}
\end{remark}

Let $f \in {\rm I}(\nu,\pi)$ and suppose we have a compact open subgroup $U_1$ of $U$, such that $\mathcal{P}_{\psi,U_1}f \in {\rm I}(\nu,\pi)_{d_P}$. Under these assumptions we can now define
\begin{equation}\label{wdef}
\begin{split}
   \lambda_\psi(\nu,\pi) f 
   &= \lim_{N_1} \, \int_{N_1} \lambda_{\psi_{w_0}} 
   (\mathcal{P}_{\psi,U_1} f(w_0^{-1}n))\opsi(n)  dn  \\
   &= \lambda_\psi(\nu,\pi) \mathcal{P}_{\psi,U_1} f \ ,
\end{split}
\end{equation}
where the limit is taken over all compact open subgroups $N_1$ of $\overline{N}$ such that $N_1 \supset \overline{N} \cap U_1$. Also, notice that the subgroup $U_1$ can be replaced by any other compact open subgroup of $U$ containing $U_1$.

All that remains is to obtain a compact open subgroup $U_1$ with the above properties. To study the analytic behavior of $\lambda_\psi(\nu,\pi)$, it is important to obtain such a group independently of $\nu$.

\begin{proposition}
   Let $K$ be a compact open subgroup of $G$. There exists a compact open subgroup $U_1 \subset U$ such that for every $\nu \in \mathfrak{a}_{M,\mathbb{C}}^*$ and every $f \in I(\nu,\pi)^K$ the function $\mathcal{P}_{\psi,U_1} f$ lies in $I(\nu,\pi)_{d_P}$.
\label{wsupport}
\end{proposition}

\begin{proof}
Let $G^*$ be the complement of $P w_0 B$ in $G$ and let $J$ be as in Lemma{~\ref{twistedj}}. The lemma says that $J=J_\psi(U)$. Hence, for any $f_J \in J$ there is a compact open subgroup $U_1$ of $U$ with
\begin{equation*}
   \int_{U_1} f_J(gu)\opsi(u)du=0, \ \forall g \in G^*.
\end{equation*}
This gives $\mathcal{P}_{\psi,U_1} f_J = 0$. Since, $I(\nu,\pi)$ is admissible, one can enlarge $U_1$ so that $\mathcal{P}_{\psi,U_1} f_J = 0$ holds for all $f_J$ in the image of the finite dimensional $I(\nu,\pi)^K$. Because of the exact sequence~(\ref{Usequence}), for any $f \in I(\nu,\pi)$ the function $\mathcal{P}_{\psi,U_1} f$ has support in $P w_0 B$, i.e., $\mathcal{P}_{\psi,U_1} f \in I(\nu,\pi)_{d_P}$.

To choose $U_1$ independently of $\nu$, let $U_2 \subset U_3 \subset \cdots $ be an exhaustive sequence of compact open subgroups of $U$. For each $n > 1$, let $\Lambda_n$ be the subset of all $\nu \in \mathfrak{a}_{M,\mathbb{C}}^*$ such that $\mathcal{P}_{\psi,U_n} f$ has support in $P w_0 B$ for all $f \in I(\nu,\pi)^K$. We have already shown that $\mathfrak{a}_{M,\mathbb{C}}^* = \cup \Lambda_n$. By Baire's lemma, at least one $\Lambda_n$ contains an open subset of $(\mathfrak{a}_\theta^*)_{\mathbb{C}}$.

We show that the complement of such a $\Lambda_n$ is also open, hence $\Lambda_n$ must be all of $(\mathfrak{a}_\theta^*)_{\mathbb{C}}$. Suppose that for some $\nu_0 \in \mathfrak{a}_{M,\mathbb{C}}^*$ there is an $f\in I(\nu_0,\pi)^K$ with support outside of $P w_0 B$. Then, by Corollary~\ref{Pnucor}, there is a $\varphi \in C_c^\infty(G,V)$ such that $\mathcal{P}_{\nu_0} \varphi = f$ and $\varphi$ is (right) $K$-invariant. Then $\mathcal{P}_\nu \varphi \in I(\nu,\pi)^K$ for all $\nu \in \mathfrak{a}_{M,\mathbb{C}}^*$. Now, there is a $g^* \in G^*$ such that $\mathcal{P}_{\nu_0}\varphi(g^*) \neq 0$ (We are using ${\rm supp} f = \left\{ g \in G \vert f(g) \neq 0 \right\}$, which can be verified for any $f \in I(\nu,\pi)$). Now, because $\varphi$, $\nu_0$ and $g^*$ are fixed, it follows from the definition that there is an $\varepsilon > 0$ such that $\mathcal{P}_{\nu_0 + \nu_\varepsilon} \varphi(g^*) \neq 0$ whenever $\left| \nu_0 - \nu_\varepsilon \right| < \varepsilon$.

\end{proof}

Notice that $U_1$ can be replaced by any other compact subgroup $U_*$ of $U$ with $U_* \supset U_1$.  

\begin{corollary}
Fix a compact open subgroup $K$, then there exists a compact open subgroup $U_* \subset U$ such that
\begin{equation*}
   \lambda_\psi(\nu,\pi)f 
   = \int_{U_*} \lambda_{\psi_{\tilde{w}_0}}(f(\tilde{w}_0^{-1}u))\opsi(u)du.
\end{equation*}
for all $f \in {\rm I}(\nu,\pi)^K$.
\end{corollary}
\begin{proof}
Choose $U_1$ as in the theorem. ${\rm I}(\nu,\pi)^K$ is finite dimensional and its elements have compact support modulo $P$. Thus, there is a compact open subgroup $U_2$ of $U$ such that $\mathcal{P}_{\psi,U_1}f$ has support in $P w_0 U_2$ for all $f \in {\rm I}(\nu,\pi)^K$. Then let $U_*$ be a compact open subgroup of $U$ containing $U_1$ and $U_2$ to get
   \begin{equation*}
      \lambda_\psi(\nu,\pi)f = \lambda_\psi(\nu,\pi) (\mathcal{P}_{\psi,U_*} f) =
      \int_{U_*} \lambda_{\psi_{\tilde{w}_0}} (f(\tilde{w}_0^{-1}n)) \opsi(n) dn.
   \end{equation*}
\end{proof}

\begin{remark}An alternative way to construct a compact open subgroup with the properties of $U_*$ can be found in Proposition~3.2 of \cite{shahidi1978}. We note that the letter \cite{casselmanletter} further comments on several interesting properties of Whittaker functionals that the author has found very useful.
\end{remark}

\begin{theorem}\label{precise:Wthm}
If $\pi = (\pi,V)$ is a generic representation of $M$ and $\nu \in \mathfrak{a}_{M,\mathbb{C}}^*$, then the induced representation ${\rm I}(\nu,\pi)$ is also generic. Explicitly, $\lambda_\psi(\nu,\pi)$ will be a Whittaker functional of ${\rm I}(\nu,\pi)$. If $\varphi \in C_c^{\infty}(G,V)$, then $\lambda_\psi(\nu,\pi) \mathcal{P}_\nu \varphi$ is a polynomial function on $q^\nu$.
\end{theorem}
\begin{proof}
   $\lambda_\psi(\nu,\pi)$ is now defined for any $f \in I(\nu,\pi)$. It is clear that it is a Whittaker functional by writing $\lambda_\psi(\nu,\pi)f = \lambda_\psi(\nu,\pi) (\mathcal{P}_{\psi,U_*} f)$ and applying Proposition~\ref{wproperty} to $\mathcal{P}_{\psi,U_*} f \in I(\nu,\pi)_{d_P}$.
   
A function $\varphi \in C_c^\infty$ is $K$-invariant for some $K$. The function $\mathcal{P}_\nu \varphi$ lies in $I(\nu,\pi)^K$, for any $\nu \in \mathfrak{a}_{M,\mathbb{C}}^*$. Let $U_1$ be as in Proposition~\ref{wsupport} and write $\varphi_\nu = \mathcal{P}_\nu \varphi$, then
   \begin{equation*}
      \lambda_\psi(\nu,\pi)(\varphi_\nu)
      = \lambda_\psi(\nu,\pi) (\mathcal{P}_{\psi,U_1} \varphi_\nu). 
   \end{equation*}
Lemma~\ref{holomorphicl} gives the required properties of $\lambda_\psi(\nu,\pi)(\varphi_\nu)$.
\end{proof}

\section{$L$-functions and intertwining operators}\label{intertwining}

We begin by recalling Shahidi's tempered $L$-function conjecture. Then, we obtain an irreducibility result for principal series representations (Lemma~\ref{tempunramInd}). We conclude the section with a discussion of intertwining operators and an assumption of Kim in \S\S~\ref{Intoper} and \ref{KimA}.

\subsection{Tempered $L$-functions and irreducibility of principal series} Shahidi's tempered $L$-function conjecture was proved in characteristic zero by H. H. Kim, in most cases \cite{k2005}. A completely local proof, that is also valid in positive characteristic, is given by V. Heiermann and E. Opdam \cite{ho2013}:

\begin{theorem}[Heiermann-Opdam, Kim]\label{temperedL}
Let $(F,\pi,\psi) \in \mathscr{L}_{\rm loc}(p)$ be tempered, then each $L(s,\pi,r_i)$ is holomorphic on $\Re(s) > 0$. 
\end{theorem}

We can combine the theory of local $L$-functions with a result of J.-S. Li \cite{li1992} to obtain the following:

\begin{lemma}\label{tempunramInd}
Let $(F,\pi,\psi) \in \mathscr{L}_{\rm loc}(p)$ be tempered and unramified. Then ${\rm I}(s,\pi)$ is irreducible for $\Re(s) > 1$.
\end{lemma}
\begin{proof}
We have that
\begin{equation*}
   \pi \hookrightarrow {\rm Ind}(\chi),
\end{equation*}
with $\chi$ an unramified unitary character of ${\bf T}(F)$. From the Satake classification, the character $\chi$ corresponds to a complex semisimple conjugacy class in the dual torus, each Satake parameter having absolute value $1$. Let
\begin{equation*}
   \chi_s = \chi \cdot q^{\left\langle {s\tilde{\alpha},H_{\bf P}(\cdot)} \right\rangle}.
\end{equation*}
From the remark following Proposition~4.1 of \cite{ks2004}, the function $\xi_\alpha(\chi_s)$ defined in \S~2 of \cite{li1992} for each non-divisible root $\alpha$ is either a non-zero constant or a factor appearing in
\begin{equation}\label{eitherL}
   \prod_{i=1}^{m_r}L(1+is,\tilde{\pi},r_i)^{-1} \text{ or } \prod_{i=1}^{m_r} L(1 - is,\pi,r_i)^{-1}.
\end{equation}
The result of Li, specifically Theorem~2.2 of [\emph{loc.\,cit.}], states that ${\rm I}(s,\pi)$ is irreducible when each $\xi_\alpha(\chi_s) \neq 0$. The local $L$-functions involved are never zero, and from Theorem~\ref{temperedL} we have that the first product $\prod_{i=1}^{m_r}L(1+is,\tilde{\pi},r_i)^{-1}$ is non-zero for $\Re(s) > 1$. We claim that the same is also true for the second product. For this, notice that because $\pi$ is tempered we can write for each $i$:
\begin{equation*}
   L(s,\pi,r_i)^{-1} = \prod_j (1 - a_{i,j} q_v^{-s}),
\end{equation*}
where the parameters $a_{i,j}$ have absolute value $1$. Then
\begin{equation*}
   \prod_{i=1}^{m_r} L(1 - is,\pi,r_i)^{-1} = \prod_{i=1}^{m_r} \prod_j (1 - a_{i,j} q_v^{-1} q_v^{is}).
\end{equation*}
Each factor in the latter product is non-zero for $\Re(is) > 1$. In particular, the product is non-zero for $\Re(s) > 1$. From Li's theorem, we must have ${\rm I}(s,\pi)$ irreducible for $\Re(s) > 1$.
\end{proof}

\subsection{Intertwining operators}\label{Intoper}

We have the following connection between the intertwining operator and Langlands-Shahidi partial $L$-functions. Let $(k,\pi,\psi,S) \in \mathscr{L}_{\rm glob}(p)$, then
\begin{equation}\label{regintL}
   {\rm M}(s,\pi,\tilde{w}_0) = \prod_{i=1}^{m_r} \dfrac{L^S(is,\pi,r_i)}{L^S(1+is,\pi,r_i)} \bigotimes_{v \in S} {\rm A}(s,\pi_v,\tilde{w}_0).
\end{equation}
The following Lemma is possible by looking into the spectral theory of Eisenstein series available over function fields.

\begin{lemma}\label{globalMESholo}
Let $(k,\pi,\psi,S) \in \mathscr{L}_{\rm glob}(p)$. If $\tilde{w}_0(\pi) \ncong \pi$, then ${\rm M}(s,\pi,\tilde{w}_0)$ and $E(s,\Phi,g,{\bf P})$ are holomorphic for $\Re(s) \geq 0$.
\end{lemma}
\begin{proof}
Let $\Phi$ be the automorphic functions of \S~1.2 that we obtain from $\pi$. We have the pseudo-Eisenstein series
\begin{equation*}
   \theta_{\Phi} = \int_{\Re(s) = s_0} E(s,\Phi,g,{\bf P}) \, ds.
\end{equation*}
This is first defined for $s_0 > \left\langle \rho_{\bf p}, \alpha^\vee \right\rangle$ by Proposition~II.1.6\,(iii) and (iv) of \cite{mw1994}. Then II.2.1~Th\'eor\`eme of [\emph{loc.\,cit.}] gives for self-associate $\bf P$ that
\begin{equation*}
   \left\langle \theta_{\Phi}, \theta_{\Phi} \right\rangle = \int_{\Re(s) = s_0} \sum_{\tilde{w} \in \left\{ 1, \tilde{w}_0 \right\}} \left\langle {\rm M}(s,\pi,\tilde{w})\Phi_{-\tilde{w}(\bar{s}\tilde{\alpha})}, \Phi_s \right\rangle \, ds,
\end{equation*}
and is zero if $\bf P$ is not self-associate. Now, for $\bf P$ self-associate, the condition $\tilde{w}_0(\pi) \ncong \pi$ allows us to shift the imaginary axis of integration by V.3.8~Lemme of \cite{mw1994} to $\Re(s) =0$. However, by IV.1.11~Proposition we have that $M(s,\pi,\tilde{w}_0)$ is holomorphic for $\Re(s) = 0$. Thus, we have that $M(s,\pi,\tilde{w}_0)$ is holomorphic for $\Re(s) > 0$. Finally, the poles of Eisenstein series are contained by the constant terms.
\end{proof}

\subsection{An assumption of Kim}\label{KimA} Locally, let $(F,\pi,\psi) \in \mathscr{L}_{\rm loc}(p)$. We have the normalized intertwining operator
\begin{equation*}
   {\rm N}(s,\pi,\tilde{w}_0) : {\rm I}(s,\pi) \rightarrow {\rm I}(s',\pi')
\end{equation*}
defined by
\begin{equation*}
   {\rm N}(s,\pi,\tilde{w}_0) = \prod_{i=1}^{m_r} \varepsilon(is,\tilde{\pi},r_i,\psi) \dfrac{L(1+is,\tilde{\pi},r_i)}{L(is,\tilde{\pi},r_i)} {\rm A}(s,\pi,\tilde{w}_0).
\end{equation*}
Globally, for $(k,\pi,\psi,S) \in \mathscr{L}_{\rm glob}(p)$, we have
\begin{equation*}
   {\rm N}(s,\pi,\tilde{w}_0) = \otimes_v \, {\rm N}(s,\pi_v,\tilde{w}_0).
\end{equation*}

The following assumption was made by H.H. Kim in his study of local Langlands-Shahidi $L$-functions and normalized intertwining operators over number fields \cite{k1999}. Having now the Langlands-Shahidi method available in positive characteristic, we are lead at this point to Kim's assumption.

\begin{assumption}\label{kimA}
Let $(k,\pi,\psi,S) \in \mathscr{L}_{\rm glob}(p)$. At every place $v$ of $k$, the normalized intertwining operator ${\rm N}(s,\pi_v,\tilde{w}_0)$ is holomorphic and non-zero for $\Re(s) \geq 1/2$.
\end{assumption}

It is already known to hold in many cases, including all the quasi-split classical groups. See \cite{kk2011} for a more detailed account. We make this assumption for the remainder of the article.

\begin{lemma}\label{prodLholo}
Let $(k,\pi,\psi,S) \in \mathscr{L}_{\rm glob}(p)$. If $\tilde{w}_0(\pi) \ncong \pi$, then
\begin{equation*}
   \prod_{i=1}^{m_r} \dfrac{L(is,\pi,r_i)}{L(1+is,\pi,r_i)}
\end{equation*}
is holomorphic for $\Re(s) \geq 1/2$.
\end{lemma}
\begin{proof}
For every place $v \notin S$, we have that
\begin{equation*}
   {\rm A}(s,\pi_v,\tilde{w}_0)f_{s,v}^0(e_v) = \prod_{i=1}^{m_r} \dfrac{L(is,\pi_v,r_i)}{L(1+is,\pi_v,r_i)} f_{s,v}^0(e_v),
\end{equation*}
for $(k_v,\pi_v,\psi_v) \in \mathscr{L}_{\rm loc}(p)$. Globally, the condition $\tilde{w}_0(\pi) \ncong \pi$ allows us to apply Lemma~\ref{globalMESholo} to $(k,\pi,\psi,S) \in \mathscr{L}_{\rm glob}(p)$. Thus ${\rm M}(s,\pi,\tilde{w}_0)$ is holomorphic for $\Re(s) \geq 0$, and in this region we have
\begin{equation*}
   {\bf M}(s,\pi,\tilde{w}_0) f_s = \prod_{i=1}^{m_r} \varepsilon(is,\tilde{\pi},r_i)^{-1} \dfrac{L(is,\tilde{\pi},r_i)}{L(1+is,\tilde{\pi},r_i)} \otimes_v {\rm N}(s,\pi_v,\tilde{w}_0) f_{s,v}^0,
\end{equation*}
where $f_s = \otimes_v \, f_{s,v}^0 \in {\rm I}(s,\pi)$. Now, with an application of Assumption~\ref{kimA}, we can prove the Lemma.
\end{proof}

\section{Global twists by characters}\label{global:poly}

We now extend a number fields result of Kim-Shahidi to the case of positive characteristic. In particular, Proposition~2.1 of \cite{ks2002} shows that, up to suitable global twists, Langlands-Shahidi $L$-functions become holomorphic. In the case of function fields, Harder's rationality result allows us to prove the stronger property of $L$-functions becoming polynomials after suitable twists.

Let $\xi \in X^*({\bf M})$ be the rational character defined by
\begin{equation*}
   \xi(m) = \det \left( {\rm Ad}(m) \vert \mathfrak{n} \right),
\end{equation*}
where $\mathfrak{n}$ is the Lie algebra of $\bf N$. At every place $v$ of $k$, we obtain a rational character $\xi_v \in X^*({\bf M})_v$. Given a gr\"ossencharakter $\chi = \otimes \chi_v : k^\times \backslash \mathbb{A}_k \rightarrow \mathbb{C}^\times$, we obtain the character
\begin{equation*}
   \chi \cdot \xi = \otimes \, \chi_v \cdot \xi_v
\end{equation*}
of ${\bf M}(\mathbb{A}_k)$. Let $(k,\pi,\psi,S) \in \mathscr{L}_{\rm glob}(p)$. For $n \in \mathbb{Z}_{\geq 0}$, form the automorphic representation
\begin{equation*}
   \pi_{n,\chi} = \pi \otimes (\chi \cdot \xi^n).
\end{equation*}

\begin{proposition}\label{twistLS}
Let $(k,\pi,\psi,S) \in \mathscr{L}_{\rm glob}(p)$ and fix a place $v_0 \in S$. Then there exist non-negative integers $n$ and $f_{v_0}$ such that for every gr\"ossencharakter $\chi = \otimes \chi_v$ with ${\rm cond}(\chi_{v_0}) \geq f_{v_0}$, we have that
\begin{equation*}
   L \left(s,\pi_{n,\chi},r_i \right),
\end{equation*}
$1 \leq i \leq m_r$, is a polynomial function on $\left\{ q^s, q^{-s} \right\}$.
\end{proposition}

\begin{proof}
In order to apply Lemmas~\ref{globalMESholo} and \ref{prodLholo}, we first need the existence of an integer $n$ such that $\tilde{w}_0(\pi_{n,\chi}) \ncong \pi_{n,\chi}$, for suitable $\chi$. For this, we follow Shahidi \cite{shahidi2000} to impose the right condition on the conductor of $\chi_{v_0}$. Write ${\bf P} = {\bf P}_\theta$, with $\theta = \Delta - \left\{ \alpha \right\}$, so that ${\bf M} = {\bf M}_\theta$ and ${\bf N} = {\bf N}_\theta$. Let ${\bf A}_\theta$ be the split torus of ${\bf M}_\theta$ and let $A_\theta = {\bf A}_\theta(k_{v_0})$, the group of $k_{v_0}$-rational points. Let
\begin{equation*}
   A_\theta^1 = \tilde{w}_0(A_\theta) A_\theta^{-1} = \left\{ a \in A_\theta \vert \tilde{w}_0(a) = a^{-1} \right\}.
\end{equation*}
Then Lemma~2 of [\emph{loc.\,cit.}] gives
\begin{equation*}
   \xi \vert_{A_\theta^1} \neq 1.
\end{equation*}
In fact, we have that $\xi: {\bf A}_\theta^1 \rightarrow {\bf G}_m$ is onto. Given an integer $n$, we choose $\ell_{v_0} \in \mathbb{Z}_{\geq 0}$ such that
\begin{equation*}
   1 + \mathfrak{p}_{v_0}^{\ell_{v_0}} \subset \xi^n(A_\theta^1)
\end{equation*}
and $\ell_{v_0}$ is minimal with this property. Let $\omega_{v_0}$ be the central character of $\pi_{v_0}$. Take
\begin{equation}\label{eq1twistLS}
  f_{v_0} \geq \max \left\{ \ell_{v_0}, {\rm cond}(\omega_{v_0}) \right\},
\end{equation}
which depends on $n$. Then
\begin{equation}\label{eq2twistLS}
   \tilde{w}_0(\omega_{v_0} \otimes (\chi_{v_0} \cdot \xi^n)) \ncong \omega_{v_0} \otimes (\chi_{v_0} \cdot \xi^n),
\end{equation}
for ${\rm cond}(\chi_{v_0}) \geq f_{v_0}$.

Let $({\bf G}_i,{\bf M}_i)$ be as in Proposition~6.4 of \cite{lsU}, where ${\bf G}_i \hookrightarrow {\bf G}$ and we have the corresponding parabolic ${\bf P}_i = {\bf M}_i {\bf N}_i$ of ${\bf G}_i$. Let $\xi_i \in X^*({\bf M}_i)$ be the rational character
\begin{equation*}
   \xi_i(m) = \det \left( {\rm Ad}(m) \vert \mathfrak{n}_i \right),
\end{equation*}
where $\mathfrak{n}_i$ is the Lie algebra of ${\bf N}_i$. There are then integers $n_1, \ldots, n_{m_r}$, such that upon restriction to ${\bf M}_i$ we have $\xi_i^{n_i} = \xi^{n_1}$ and $\chi \cdot \xi_i^{n_i} = \chi \cdot \xi^{n_1}$. With this choice of $n = n_1$, we choose $f_{v_0}$ as in \eqref{twistLS} at the place $v_0 \in S$. Then, for any gr\"ossencharakter $\chi : k^\times \backslash \mathbb{A}_k^\times \rightarrow \mathbb{C}^\times$ with ${\rm cond}(\chi_{v_0}) \geq f_{v_0}$. Equation \eqref{eq2twistLS} at $v_0$ ensures that
\begin{equation*}
   \tilde{w}_0(\pi_{n,\chi}) \ncong \pi_{n,\chi}.
\end{equation*}

Now, Lemma~\ref{globalMESholo} together with \eqref{eq1rationalL} give that
\begin{equation*}
   \prod_{i=1}^{m_r} L^S(1+is,\pi_{n,\chi},r_i)
\end{equation*}
is holomorphic and non-zero for $\Re(s) \geq 0$. The induction step found in \S~6.2 of \cite{lsU} allows us to isolate each $L$-function and conclude that each
\begin{equation*}
   L^S(s,\pi_{n,\chi},r_i)
\end{equation*}
is holomorphic and non-zero on $\Re(s) \geq 1$. Furthermore, with the aid of Lemma~\ref{prodLholo} we conclude that each $L^S(s,\pi_{n,\chi},r_i)$ is holomorphic on $\Re(s) \geq 1/2$.

The functional equation of Theorem~11.1 of \cite{lsU}, allows us to conclude holomorphy for $\Re(s) \leq 1/2$. The automorphic $L$-functions $L(s,\pi_{n,\chi},r_i)$ being now entire, in addition to being rational by Theorem~\ref{rationalL}, must be polynomials on $\left\{ q^{-s}, q^s \right\}$.
\end{proof}

The following corollary is obtained from the proof of the Proposition.

\begin{corollary}\label{Lpoly}
Let $(k,\pi,\psi,S) \in \mathscr{L}_{\rm glob}(p)$ be such that $\tilde{w}_0(\pi) \ncong \pi$. Then $L(s,\pi,r_i)$, $1 \leq i \leq m_r$, is a polynomial on $\left\{ q^{-s}, q^s \right\}$.
\end{corollary}

\section{Transfer and the generic unitary dual of classical groups}\label{transfer}

We recall the results of Ganapathy on the Kazhdan transfer between close local fields and the Langlands-Shahidi local coefficient \cite{Ga2015}. We use this to transfer a result of Lapid, Mui\'c and Tadi\'c on unitary generic representations of the classical groups \cite{LaMuTa2004}. In particular, for the split classical groups, Ganapathy and Varma show that local L-functions and $\gamma$-factors are compatible with the Kazhdan transfer \cite{GaVaJIMJ}; we observe that this holds in more generality (Theorem~\ref{transfer:L3}). An alternative approach to the results of this section would be to use the techniques of \cite{AuBaPlSopreprint, BaHeLeSe2010}.

\subsection{Kazhdan transfer} Let $F$ and $F'$ be non-archimedean local fields that are $m$-close in the sense of Kazhdan \cite{Ka1986}, i.e, there is an isomorphism of rings $\mathcal{O}_F / \mathfrak{p}_F^m \cong \mathcal{O}_{F'} / \mathfrak{p}_{F'}^m$. Let ${\bf G}$ be a split connected reductive group scheme defined over $\mathbb{Z}$. Let $\mathcal{I}$ denote the standard Iwahori subgroup scheme of ${\bf G}$ and let
\begin{equation*}
   \mathcal{I}_m = {\rm Ker}\left\{ {\mathcal{I}}(\mathcal{O}_F) \rightarrow {\mathcal{I}}(\mathcal{O}_F/\mathfrak{p}_F^m) \right\}.
\end{equation*}
We use the results of Ganapathy, who establishes a refined Kazhdan tranfer for the Hecke algebra $\mathscr{H}(G,\mathcal{I}_m)$ in \cite{Ga2015}.

Let $\mathfrak{R}(G)$ be the category of smooth complex representations of $G$. Let $\mathfrak{R}^m(G)$ be the subcategory of representations $(\pi,V)$ of $G$ such that $\pi(G)(V^{\mathcal{I}_m}) = V$. If $\mathscr{H}(G,\mathcal{I}_m)$-mod is the category of $\mathscr{H}(G,\mathcal{I}_m)$-modules, Proposition~3.16 of \cite{Ga2015} shows that $\mathfrak{R}^m(G)$ is closed under sub-quotients and gives a functor
\begin{equation*}
   \mathfrak{R}^m(G) \rightarrow \mathscr{H}(G,\mathcal{I}_m)\text{-mod},
\end{equation*}
which is an equivalence of categories.

Ganapathy constructs a Kazhdan isomorphism
\begin{equation*}
   \mathscr{H}(G,\mathcal{I}_m) \longrightarrow \mathscr{H}(G',\mathcal{I}_m')
\end{equation*}
which also gives a bijection
\begin{equation*}
   \left\{ \begin{array}{c} \text{Isomorphism classes of} \\
  				     \text{irreducible admissible } \pi \\
  				     \text{of } G \text{ with } \pi^{\mathcal{I}_m}\neq 0
					\end{array} \right\}
   \xrightarrow{\zeta_m}
   \left\{ \begin{array}{c} \text{Isomorphism classes of} \\
  				     \text{irreducible admissible } \pi' \\
  				     \text{of } G' \text{ with } \pi^{\mathcal{I}_m'}\neq 0
					\end{array} \right\}.
\end{equation*}
We use the notation of \cite{GaVaJIMJ}, to denote $\pi \sim_m \pi'$ whenever we are in the above situation. A property that is preserved under the transfer is the Moy-Prasad depth \cite{MoPr1994}, ${\rm depth}(\pi) = {\rm depth}(\pi')$. We refer to \cite{Ga2015} for more details, and also \S~13 of \cite{GaVaJIMJ} for a summary of results.

\begin{remark}
Kazhdan \cite{Ka1986}, produces an isomorphism of Hecke algebras using principal congruence subgroups ${\bf\mathcal{K}}_m$ of $G$, for $m$ large enough, instead of the Iwahori filtration. Since $\mathcal{I}_m \subset\mathcal{K}_m$, we have that $\mathscr{H}(G,\mathcal{K}_m) \subset \mathscr{H}(G,\mathcal{I}_m)$. Corollary~3.15 of \cite{Ga2015}, proves that $\zeta_m \vert_{\mathscr{H}(G,\mathcal{K}_m)}$ is the Kazhdan isomorphism between $\mathscr{H}(G,\mathcal{K}_m)$ and $\mathscr{H}(G,\mathcal{K}_m')$ for appropriate $m$.
\end{remark}

\subsection{Parabolic induction and transfer} Fix a pair $({\bf G},{\bf M})$ of split reductive group schemes such that ${\bf M}$ is a Levi component of a parabolic subgroup $\bf P$ of $\bf G$. We have the transfer of the previous section for $\bf M$, so that given $m$-close local fields $F$ and $F'$, the map $\zeta_m$ gives a transfer $\sigma \sim_m \sigma'$ between admissible representations of $M$ and $M'$. The results of \S~4 and \S~5 of \cite{Ga2015} give a map $\kappa_m$ for the corresponding parabolically induced representations. The Kazhdan transfer is compatible with parabolic induction. For this, suppose that we have $\pi^{\mathcal{I}_m} \neq 0$, then

\begin{center}
\begin{tikzpicture}
   \draw (0,2) node {${\rm Ind}(\pi)$ of $G$};
   \draw[->] (0,.5) -- (0,1.5);
   \draw (0,1) node [left] {${\rm I}_P^G$};
   \draw (0,0) node {$\pi$ of $M$};
   \draw[<->] (1.5,2) -- (3.25,2);
   \draw (2.35,2) node[above] {$\kappa_,$};
   \draw (5,0) node {$\pi'$ of $M'$};
   \draw (5,2) node {${\rm Ind}(\pi')$ of $G'$};
   \draw[<->] (1.5,0) -- (3.25,0);
   \draw (5,1) node [right] {${\rm I}_{P'}^{G'}$};
   \draw (2.35,0) node[above] {$\zeta_{m+3}$};
   \draw[<-] (5,1.5) -- (5,.5);
\end{tikzpicture}
\end{center}

Suppose we are given two characters $\psi : F \rightarrow \mathbb{C}^\times$ and $\psi' : F' \rightarrow \mathbb{C}^\times$. We write $\psi \sim_m \psi'$ if ${\rm cond}(\psi) = {\rm cond}(\psi') = t$ and
\begin{equation*}
   \psi \vert_{\mathfrak{p}_F^{t-m}/\mathfrak{p}_F^t} \cong 
   \psi' \vert_{\mathfrak{p}_{F'}^{t-m}/\mathfrak{p}_{F'}^t}.
\end{equation*}

An important result is the following: Let $\pi$ be a generic representation of $G_n$ and $\tau$ a smooth irreducible representation of ${\rm GL}_r(F)$. Let $m \geq 1$ be such that ${\rm depth}(\pi)$, ${\rm depth}(\tau) < m$. There exists an integer $l$, depending on $n,r,m$, such that for each $F'$ that is $l$-close to $F$, we have
\begin{equation}
   \gamma(s,\pi \times \tau,\psi) = \gamma(s,\pi' \times \tau',\psi'),
\end{equation}
if $\pi \sim_l \pi'$, $\tau \sim_l \tau'$ and $\psi \sim_l \psi'$. This is Proposition~13.5.3 of \cite{GaVaJIMJ} for the split classical groups. For a general split reductive group scheme, we observe that we have:

\begin{theorem}\label{transfer:L3}
Let $(F',\pi',\psi') \in \mathscr{L}_{\rm loc}(p)$ be a generic representation of $M' = {\bf M}(F')$. Let $m \geq 1$ be such that ${\rm depth}(\pi')<m$. There exists an integer $l$, depending on $m$, such that for every non-archimedean local field $F$ that is $l$-close to $F'$ we have
\begin{align*}
   L(s,\pi,r_i) &= L(s,\pi',r_i) \\
   \varepsilon(s,\pi,r_i,\psi) &= \varepsilon(s,\pi',r_i,\psi'),
\end{align*}
if $\pi \sim_l \pi'$ and $\psi \sim_l \psi'$, $1 \leq i \leq m_r$.
\end{theorem}
\begin{proof}
First assume that
\begin{equation}\label{transfer:gamma}
   \gamma(s,\pi,r_i,\psi) = \gamma(s,\pi',r_i,\psi')
\end{equation}
is true when $\pi \sim_l \pi'$ are tempered representations. Then, the construction of $L$-functions and $\varepsilon$-factors of \cite{shahidi1990} and \cite{lsU}, builds the general case from tempered representations aided by Langlands classification. The theorem then follows, since the Kazdan transfer preserves temperedness, is compatible with parabolic induction and twists by unramified characters (see \S~4 of \cite{Ga2015}).

We actually prove \eqref{transfer:gamma} for every triple $(F',\pi',\psi') \in \mathscr{L}_{\rm loc}(p)$. This is done by an induction argument on the number of irreducible constituents of the adjoint action of ${}^LM$ on ${}^L\mathfrak{n}$, namely $r = \oplus_{i=1}^{m_r} r_i$. If $r = r_1$ is irreducible, we have that
\begin{equation*}
   \gamma(s,\pi',r_1,\psi') = C_{\psi'}(s,\pi',\tilde{w}_0').
\end{equation*}
If $m_r > 1$, we can use Lemma~4.4 of \cite{lsU}. Then for each $i$, $1 \leq i \leq m_r$, there exists a pair $({\bf G_i},{\bf M_i})$, such that the adjoint action $r'$ of ${}^LM_i$ on ${}^L\mathfrak{n}_i$ has $r' = \oplus_{j=1}^{m_{r'}} r_j'$, with $m_{r'} < m_r$ and $r_i = r_1'$. We continue in this way until we find a pair $({\bf G}_{i_k},{\bf M}_{i_k})$, such that $r_i = r_{1}^{(k)}$ and the adjoint action after $k$ steps $r^{(k)} = r_1^{(k)}$ is irreducible. From the way the construction is made, there is a triple $(F',\pi_{(k)}',\psi') \in \mathscr{L}({\bf G}_{i_k},{\bf M}_{i_k},p)$, for which we have that
\begin{equation*}
   \gamma(s,\pi',r_i,\psi') = C_{\psi'}(s,\pi_{(k)}',\tilde{w}_0').
\end{equation*}
Finally, thanks to Theorem~5.5 of \cite{Ga2015}, we have that there is a sufficiently large $l$, so that for every non-archimedean local field $F$, if we let $\pi \sim_l \pi'$ and $\psi \sim_l \psi'$, then
\begin{equation*}
   C_\psi(s,\pi_{(k)},\tilde{w}_0) = C_{\psi'}(s,\pi_{(k)}',\tilde{w}_0').
\end{equation*}
From this, equation~\eqref{transfer:gamma} follows.
\end{proof}

\subsection{Generic unitary dual}\label{generic:unitary} Assume $F$ has characteristic zero. Let $\pi$ be a (unitary) generic representation of $G_n = {\bf G}_n(F)$. From \cite{LaMuTa2004} we see that $\pi$ is a constituent of an induced representation of the form
\begin{equation}\label{LaMuTa:bound}
   {\rm Ind}(\delta_{1} \nu^{\beta_1} \otimes \cdots \otimes \delta_{d} \nu^{\beta} \otimes
   \rho_{1} \nu^{\alpha_1} \otimes \cdots \otimes \rho_{e}' \nu^{\alpha_e} \otimes \pi_{0}).
\end{equation}
The representations $\delta_{i}$ and $\rho_{j}$ are unitary discrete series; the Langlands parameters are of the form $1 > \beta_1 > \cdots \beta_d > 1/2 > \alpha_1 > \cdots > \alpha_e > 0$; and, $\pi_{0}$ is tempered generic of a classical group $G_{n_0}$ of the same kind as $G_n$. In fact, except for the case of ${\bf G}_n = {\rm SO}_{2n}$, the set $\left\{ \beta_1, \ldots, \beta_d \right\}$ is empty and the Langlands parameters are bounded by $1/2$.

\begin{theorem}\label{uclassical:p:bound}
Let $F'$ be a non-Archimedean local field of characteristic $p$ and let $\pi'$ be a (unitary) generic representation of $G'_n = {\bf G}_n(F')$. Then $\pi' \sim_l \pi$ is a constituent of an induced representation of the form
\begin{equation}
   {\rm Ind}(\delta'_{1} \nu^{\beta_1} \otimes \cdots \otimes \delta'_{d} \nu^{\beta} \otimes
   \rho'_{1} \nu^{\alpha_1} \otimes \cdots \otimes \rho'_{e} \nu^{\alpha_e} \otimes \pi_{0}'),
\end{equation}
with the notation of \eqref{LaMuTa:bound}.
\end{theorem}

\begin{proof} Let us show that unitarity is preserved under the Kazhdan transfer. Let $F$ be or characteristic $0$ and such that it is $l$-close to $F'$. Take $l$ sufficiently large, and let $\pi$ be a representation of $G$, such that $\pi' \sim_l \pi$ and ${\rm depth}(\pi') = {\rm depth}(\pi) < m$. This is possible by the results of \S~3 of [\emph{loc.\,cit.}], where $l$ depends on $m$.

Since $(\pi,V')'$ is unitary, there is a positive definite Hermitian form $h'$ on $V'$. For every positive integer $r > l$, we obtain a positive definite Hermitian form $h'_r$ on $V'^{\mathcal{I}_r}$, defined by $h'_r(v',w') = h'(v',w')$ for $v',w' \in V'^{\mathcal{I}_r}$. We can transport the structure, using the Kazhdan isomorphism $\zeta_r$ to $V^{\mathcal{I}_r}$. We thus obtain a positive definite hermitian form $h_r$ on $V^{\mathcal{I}_r}$. Let $t>r>l$, then at every stage we have compatibility
\begin{center}
\begin{tikzpicture}
   \draw (0,0) node {$V^{\mathcal{I}_r}$};
   \draw (0,.5) node {$\cup$};
   \draw (0,1) node {$V^{\mathcal{I}_t}$};
   \draw[<->] (.5,1) -- (2.5,1);
   \draw (1.5,1) node[above] {$\zeta_t$};
   \draw (3,1) node {$V'^{\mathcal{I}_t'}$};
   \draw (3,0) node {$V'^{\mathcal{I}_r'}$};
   \draw[<->] (.5,0) -- (2.5,0);
   \draw (1.5,0) node[below] {$\zeta_r$};
   \draw (3,.5) node {$\cup$};
\end{tikzpicture}
\end{center}
for the Hermitian forms. We can then define a hermitian form $h$ on $V$. Let $v,w \in V \setminus \left\{ 0 \right\}$, then there exists an $r > l$ such that $v,w \in V^{\mathcal{I}_r}$. Then $h(v,w) = h_r(v,w) >0$. Hence the form is positive definite and $\pi$ is unitary.

Genericity is respected by the Kazhdan transfer, this is the content of \S~4.1 of \cite{Ga2015}. The representation $\pi$ of $G$ is then unitary generic. Since we are now in characteristic $0$, we have the Lapid-Mui\'c-Tadic result, hence $\pi$ is a constituent of an induced representation of the form given by equation~\eqref{LaMuTa:bound}. With the notation for this setting for $\pi$, we can take $\pi_0' \sim_l \pi_0$, $\delta_i' \sim_l \delta_i$, $\rho_j' \sim_l \rho_j$, for $1< i < d$, $1 < j < e$, all with depth less than $m$ and $l$ large enough.

The Kazhdan transfer behaves well under parabolic induction, see \S~4 of \cite{Ga2015}. It preserves discrete series representations, tempered representations, and we can twist by unramified characters. Thus the Lapid-Mui\'c-Tadi\'c presentation for $\pi$ transfers to $\pi'$, and we obtain the desired form.
\end{proof}

\section{On the holomorphy of $L$-functions for the classical groups}\label{classical:holo}

Let ${\bf G}_n$ be either ${\rm SO}_{2n+1}$, ${\rm Sp}_{2n}$, ${\rm SO}_{2n}$ or a quasi-split unitary group ${\rm U}_N$ (with $N = 2n+1$ or $2n$). We leave the study of ${\rm GSpin}_N$ groups, in addition to non-split quasi-split ${\rm SO}_{2n}^*$ and ${\rm GSpin}_{2n}^*$, for another occasion. Our groups are defined over a global function field $k$. In the case of a quasi-split unitary group, there is a degree-$2$ extension defining the hermitian form. To include this case in a uniform way, we make the following convention: let $K = k$ if ${\bf G}_n$ is split, and $K/k$ denotes a separable quadratic extension if ${\bf G}_n$ is a quasi-split unitary group. The group ${\rm Res}\,{\rm GL}_m$ denotes the usual general linear group ${\rm GL}_m$ in the split case and is obtained via restriction of scalars in the non-split quasi-split case.

\subsection{Exterior square, symmetric square and Asai $L$-functions}

We let $\bf M$ be the Siegel Levi subgroup of ${\bf G}_n$ as in \S~2.2 of \cite{lomeliSiegel}. In the case of a quasi-split unitary group, we locally have that a place $v$ of $k$ may be split. In this case we set $K_v = k_v \times k_v$. In particular, we let $\mathscr{L}_{\rm loc}(p,{\bf G},{\bf M})$ be the class of triples $(K_v/k_v,\pi_v,\psi_v)$ consisting of: a non-archimedean local field $k_v$; $K_v = k_v$ if ${\bf G}_n$ is split; $K_v$ a degree-2 finite \'etale algebra over $k_v$ if ${\bf G}_n$ is a quasi-split unitary group; an irreducible generic representation $\pi$ of ${\bf M}(k_v)$; and, a non-trivial continuous character $\psi_v$ of $k_v$. In our setting, the triples $(K_v/k_v,\pi_v,\psi_v)$ always arise from the corresponding global objects. More precisely, we use the notation $\mathscr{L}_{\rm glob}(p,{\bf G},{\bf M})$ to denote the corresponding global class of quadruples $(K/k,\pi,\psi,S)$.

In these cases, $L$-functions and related local factors arise from representations of ${\rm GL}_n$. Let $\rho_n$ denote the standard representation of ${\rm GL}_n(\mathbb{C})$. The adjoint action is either irreducible, and we write $r = r_1$ in the following cases
\begin{equation}\label{Siegelirred}
   r = \left\{ \begin{array}{ll}
   		{\rm Sym}^2 \rho_n & \text{ if } {\bf G} = {\rm SO}_{2n+1} \\
		\wedge^2 \rho_n & \text{ if } {\bf G} = {\rm SO}_{2n} \\
		r_\mathcal{A} & \text{ if } {\bf G} = {\rm U}_{2n}
	\end{array}, \right.
\end{equation}
or reducible, when we have $r = r_1 \oplus r_2$ as follows
\begin{equation}\label{Siegelred}
   r = \left\{ \begin{array}{ll}
		\rho_n \oplus \wedge^2 \rho_n & \text{ if } {\bf G} = {\rm Sp}_{2n} \\
		(\rho_n \otimes \rho_1) \oplus (r_\mathcal{A} \otimes \eta_{K/k}) & \text{ if } {\bf G} = {\rm U}_{2n+1}
	\end{array} . \right.
\end{equation}
We refer to \cite{lomeliSiegel} for any unexplained notation concerning the Siegel Levi case.

All of these cases share the advantage that the Ramanujan conjecture holds for ${\rm GL}_n$ over function fields thanks to the work of L. Lafforgue \cite{llafforgue2002}. We can actually prove a general theorem for any pair $({\bf G},{\bf M})$ under the assumption that $\pi$ satisfies the Ramanujan conjecture, which we now recall. 

\begin{conjecture}
Let $\pi$ be a globally generic cuspidal automorphic representation of a quasi-split connected reductive group $\bf H$. Then each $\pi_v$ is tempered.
\end{conjecture}

\begin{remark}\label{Ramrem}
In Section~\ref{classicalLholo} below, we only make use of Corollary~\ref{extsymAholo}, where the assumption of the next proposition holds. We do not assume the Ramanujan conjecture for the classical groups to prove Theorem~\ref{globholo}.
\end{remark}

\begin{proposition}\label{Ramholo}
Let $(K/k,\pi,\psi,S) \in \mathscr{L}_{\rm glob}(p,{\bf G},{\bf M})$, with $\bf M$ a maximal Levi subgroup of a quasi-split connected reductive group $\bf G$. Assume that $\pi$ satisfies the Ramanujan conjecture. Then, for each $i$, $1 \leq i \leq m_r$, the automorphic $L$-function $L(s,\pi,r_i)$ is holomorphic for $\Re(s) > 1$.
\end{proposition}
\begin{proof}
Let $(K/k,\pi,\psi,S) \in \mathscr{L}_{\rm glob}(p,{\bf G},{\bf M})$. From Theorem~\ref{temperedL}, we know that at places $v \in S$ each of the $L$-functions $L(s,\pi_v,r_{i,v})$ is holomorphic for $\Re(s) > 0$.
We begin with the known observation that local components of residual automorphic representations are unitary representations. Furthermore, if the local representation ${\rm I}(s,\pi_v)$ is irreducible, then it cannot be unitary. From Lemma~\ref{tempunramInd}, we conclude that the global intertwining operator ${\rm M}(s,\pi,\tilde{w}_0)$ must be holomorphic for $\Re(s) > 1$. Then, from equation~\eqref{regintL} we conclude that the product
\begin{equation*}
   \prod_{i = 1}^{m_r} \dfrac{L^S(is,\pi,r_i)}{L^S(1+is,\pi,r_i)}
\end{equation*}
is holomorphic on $\Re(s) > 1$. Since the poles of Eisenstein series are contained in the constant terms, with an application of equation~\eqref{eq1rationalL} we can further conclude that
\begin{equation*}
   \prod_{i=1}^{m_r} L^S(1+is,\pi,r_i)^{-1}
\end{equation*}
is holomorphic for $\Re(s) > 1$.
Now, the induction step found in \S~6 of \cite{lsU} allows us to isolate each $L$-function and prove that each
\begin{equation*}
   L^S(s,\pi,r_i)
\end{equation*}
is holomorphic for $\Re(s) > 1$.
\end{proof}


\begin{corollary}\label{extsymAholo}
Let $(K/k,\tau,\psi,S) \in \mathscr{L}_{\rm glob}(p,{\bf G},{\bf M})$, with $\bf M$ the Siegel Levi subgroup of a quasi-split classical group. Then, for each $i$, $1 \leq i \leq m_r$, the automorphic $L$-function $L(s,\tau,r_i)$ is holomorphic and non vanishing for $\Re(s) > 1$.
\end{corollary}
\begin{proof}
The assumptions of Proposition~\ref{Ramholo} are satisfied as noted in Remark~\ref{Ramrem}, which shows the holomorphy result. To prove non vanishing, we look at all possibilities for $r_i$ given by equations~\eqref{Siegelirred} and \eqref{Siegelred}. The $L$-functions corresponding to $\rho_n$ and $\rho_n \otimes \rho_1$ are understood and known to satisfy the properties of the corollary \cite{js1976, js1981}. The $L$-functions $L(s,\tau,r)$ for $r = {\rm Sym}^2$, $\wedge^2$, $r_\mathcal{A}$ or $r_\mathcal{A} \otimes \eta_{K/k}$ satisfy the following relationships
\begin{align*}
   L(s,\tau \times \tau) &= L(s,\tau,{\rm Sym}^2) L(s,\tau,\wedge^2) \\
   L(s,\tau \times \tau^\theta) &= L(s,\tau,r_{\mathcal{A}}) L(s,\tau,r_{\mathcal{A}} \otimes \eta_{K/k}).
\end{align*}
Each $L$-function $L(s,\tau,r)$ appearing on the right hand sides of the previous two equations is holomorphic for $\Re(s) > 1$. The Rankin-Selberg product $L$-functions on the left hand sides are non vanishing for $\Re(s) > 1$ \cite{js1981}. Hence, each $L(s,\tau,r)$ in turn must be non vanishing for $\Re(s) > 1$.
\end{proof}

\subsection{Rankin-Selberg $L$-functions for the classical groups}\label{classicalLholo}
We begin by adapting the notation to the classical groups. Let $\mathscr{L}_{\rm glob}(p,{\bf G}_n,{\rm Res}\,{\rm GL}_m)$ be the class consisting of quintuples $(K/k,\pi,\tau,\psi,S)$: $k$ a global function field of characteristic $p$; $K = k$ if ${\bf G}_n$ is split and $K/k$ a separable quadratic extension if ${\bf G}_n$ is a quasi-split unitary group; $\pi = \otimes' \pi_v$ a globally generic cuspidal automorphic representation of ${\bf G}_n(\mathbb{A}_k)$; $\tau = \otimes' \tau_v$ a cuspidal (unitary) automorphic representation of ${\rm Res}\,{\rm GL}_m(\mathbb{A}_k) = {\rm GL}_m(\mathbb{A}_K)$; $\psi : k \backslash \mathbb{A}_k \rightarrow \mathbb{C}^\times$ a global additive character; and $S$ a finite set of places of $k$ such that $\pi_v$ and $\psi_v$ are unramified for $v \notin S$. In the case of a separable quadratic extension, we can think of the additive character $\psi_K: K \backslash \mathbb{A}_K \rightarrow \mathbb{C}^\times$ obtained via the trace.

\begin{theorem}\label{globholo}
Let $(K/k,\pi,\tau,\psi,S) \in \mathscr{L}_{\rm glob}(p,{\bf G}_n,{\rm Res}\,{\rm GL}_m)$. Then $L(s,\pi \times \tau)$ is holomorphic for $\Re(s) > 1$ and has at most a simple pole at $s = 1$.
\end{theorem}
\begin{proof}
Thanks to the work of L. Lafforgue \cite{llafforgue2002}, each local component $\tau_v$ of the cuspidal unitary $\tau$ arises from an induced representation of the form
\begin{equation*}
   {\rm Ind}(\tau_{1,v} \otimes \cdots \otimes \tau_{f,v}),
\end{equation*}
with each $\tau_{l,v}$ tempered. Also, the classification of generic unitary representations of ${\bf G}_n(k_v)$ gives that every $\pi_v$ is a constituent of
\begin{equation}\label{eq1globholo}
   {\rm Ind}(\delta_{1,v} \nu^{\beta_1} \otimes \cdots \otimes \delta_{d,v} \nu^{\beta_d} \otimes
   \delta_{1,v}' \nu^{\alpha_1} \otimes \cdots \otimes \delta_{e,v}' \nu^{\alpha_e} \otimes \pi_{0,v}).
\end{equation}
Here, the notation is that of \eqref{LaMuTa:bound} and Theorem~\ref{uclassical:p:bound} for the split classical groups. For the quasi-split unitary groups see Remark~\ref{rmk:ubound}.

Then, the multiplicativity property of Langlands-Shahidi $L$-functions gives
\begin{align*}
   L(s,\pi_v \times \tau_v) = L(s,\pi_{v,0} \times \tau_v) & \prod_{i=1}^d \prod_{j=1}^e \prod_{l=1}^f L(s + \beta_i, \delta_{i,v} \times \tau_{l,v}) L(s + \alpha_j, \delta_{j,v}' \times \tau_{l,v}) \\
   \times &\prod_{i=1}^d \prod_{j=1}^e \prod_{l=1}^f L(s - \beta_i, \delta_{i,v} \times \tau_{l,v}) L(s - \alpha_j, \delta_{j,v}' \times \tau_{l,v}).
\end{align*}
From Theorem~\ref{temperedL}, each of the $L$-functions appearing in the right hand side is holomorphic for $\Re(s)$ large enough. This carries through to the left hand side and we can conclude that $L(s,\pi_v \times \tau_v)$ is holomorphic for $\Re(s) > \beta_1$. In particular, for $\Re(s) > 1$.

Now, we globally let $\sigma = \tau \otimes \tilde{\pi}$, so that $(K/k,\sigma,\psi,S) \in \mathscr{L}_{\rm glob}(p,{\bf G},{\bf M})$. Where ${\bf G}$ is a classical group of rank $l = m+n$ of the same type as ${\bf G}_n$ and ${\bf M} = {\rm Res}\,{\rm GL}_m \times {\bf G}_n$ is a maximal Levi subgroup. 
As observed in the number fields case by H. H. Kim in \cite{kim2000}, and from \cite{mw1994, morris1982} over function fields, if the global intertwining operator ${\rm M}(s,\Phi,g,{\bf P})$ has a pole at $s_0$, then a subquotient of ${\rm I}(s_0,\sigma)$ would belong to the residual spectrum and we would have that almost every ${\rm I}(s_0,\sigma_v)$ is unitary.

However, to obtain a contradiction, we claim for $\Re(s) > 1$ that the representation ${\rm I}(s,\sigma_v)$ cannot be unitary for at least one $v \notin S$ (we can actually show this for all $v \notin S$). For this, we begin by fixing a place $v \notin S$, which we assume remains inert if we are in the case of a non-split quasi-split classical group (see Remark~\ref{rmk:ubound}). In these cases we now apply equation~\eqref{eq1globholo} for the group ${\bf G}_l$. If ${\rm I}(s_0,\sigma_v)$ were unitary then, up to rearrangement if necessary, it would be of the form
\begin{equation}
   {\rm Ind}(\chi_{1,v}\nu^{s_0} \otimes \cdots \otimes \chi_{m,v}\nu^{s_0} \otimes \mu_{1,v} \nu^{\beta_1} \otimes \cdots \otimes \mu_{d,v} \nu^{\beta_d} \otimes
   \mu_{1,v}' \nu^{\alpha_1} \otimes \cdots \otimes \mu_{e,v}' \nu^{\alpha_e} \otimes \mu_{0,v}),
\end{equation}
where we now have unramified unitary characters $\mu_{0,v}$, $\mu_{i,v}$ and $\mu_{j,v}'$; the character $\mu_{0,v}$ is taken to be trivial unless we are in the odd unitary group case. The Langlands parameters are of the form $1 > \beta_1 > \cdots \beta_d > 1/2 > \alpha_1 > \cdots > \alpha_e > 0$. The classification tells us that this cannot be the case if $\Re(s_0) > 1$. Hence, the global intertwining operator ${\rm M}(s,\Phi,g,{\bf P})$ must be holomorphic for $\Re(s) > 1$.

The product
\begin{equation*}
   \prod_{i = 1}^{m_r} \dfrac{L^S(is,\sigma,r_i)}{L^S(1+is,\sigma,r_i)}
\end{equation*}
is then holomorphic on $\Re(s) > 1$. For the classical groups we precisely have that $m_r = 2$ and
\begin{equation}\label{eitherL}
   \prod_{i=1}^{m_r}L^S(s,\sigma,r_i) = L^S(s,\sigma,r_{1}) L^S(2s,\tau,r_{2}),
\end{equation}
where $r_{1} = \rho_n \otimes \tilde{\rho}_m$ and $L^S(s,\sigma,r_2) = L^S(s,\tau,r_{2})$ has $r_{2}$ either an exterior square, symmetric square or Asai $L$-function. The induction step is given by the Siegel Levi case of Corollary~\ref{extsymAholo}. We can then cancel the second $L$-functions and conclude that the quotient
\begin{equation}\label{eq2globholo}
   \dfrac{L^S(s,\pi \times \tau)}{L^S(1 + s,\pi \times \tau)}
\end{equation}
is holomorphic for $\Re(s) > 1$.

Finally, the poles of Eisenstein series are contained in the constant terms, hence equation~\eqref{eq1rationalL} gives that $L^S(1 + s,\pi \times \tau)^{-1}$ is holomorphic for $\Re(s) >1$. In this way, we can cancel the now non-zero denominator in equation~\eqref{eq2globholo} to conclude that $L^S(s,\pi \times \tau)$ is holomorphic on $\Re(s) >1$. We have also noted that $L(s,\pi_v \times \tau_v)$ is holomorphic for $\Re(s) > 1$ at every place $v \in S$. We then obtain the required holomorphy property for the completed $L$-function $L(s,\pi \times \tau)$.

Now, we have that the poles of the global intertwining operator $M(s,\sigma,\tilde{w}_0)$ are all simple \cite{mw1994}. Then, by equation~\eqref{regintL}, the quotient
\begin{equation*}
   \dfrac{L^S(s,\pi \times \tau)L^S(2s,\tau,r_2)}{L^S(1+s,\tilde{\pi} \times \tilde{\tau})L^S(1+2s,\tilde{\tau},r_2)}
\end{equation*}
has at most a simple pole at $s = 1$. From Corollary~\ref{extsymAholo}, $L^S(2,\tau,r_2)$ and $L^S(3,\tilde{\tau},r_2)$ are different from zero. Thus, $L(s,\pi \times \tau)$ has at most a simple pole at $s = 1$.
\end{proof}

\begin{remark}\label{rmk:ubound} The case of unitary groups can also be established by choosing a split place $v_0$, and using the known result for ${\rm GL}_n$. This is because the argument only requires one place $v_0$ where ${\rm I}(s,\sigma_{v_0})$ cannot be unitary for $\Re(s) > 1$. Each representation $\pi_{w_i}$, $i=1$ or $2$, $w_i \vert v_0$ is known to have Langlands parameters $0 \leq t_{i,d_i} \leq \cdots \leq t_{i,1}< 1/2$.
\end{remark}

\section{On functoriality for the classical groups}\label{classical:fff}

Throughout this section, we restrict ourselves to the split case. Theorems~\ref{RR} and \ref{fff} were established in \cite{lomeli2009, ls} under the assumption ${\rm char}(k) \neq 2$; we now remove this restriction. The corresponding results for the quasi-split unitary groups are proved in \cite{lsU}.

\begin{theorem}\label{RR}
Let ${\bf G}_n$ be a split classical group defined over a function field $k$.
\begin{itemize}
   \item[(i)] \emph{(Ramanujan Conjecture).} If $\pi = \otimes' \pi_v$ is a globally generic cuspidal automorphic representation of ${\bf G}_n(\mathbb{A}_k)$, then each local component $\pi_v$ is tempered.
\end{itemize}
Let $(k,\pi,\tau,\psi) \in \mathscr{L}_{\rm glob}(p,{\bf G}_n,{\bf G}_m)$, with ${\bf G}_m$ either ${\rm GL}_m$ or a split classical group of rank $m$. Then
\begin{itemize}
   \item[(ii)] \emph{(Rationality).} $L(s,\pi \times \tau)$ converges absolutely on a right half plane and has a meromorphic continuation to a rational function on $q^{-s}$.
   \item[(iii)] \emph{(Functional Equation).} $L(s,\pi \times \tau) = \varepsilon(s,\pi \times \tau) L(1-s,\tilde{\pi} \times \tilde{\tau})$.
   \item[(iv)] \emph{(Riemann Hypothesis).} The zeros of $L(s,\pi \times \tau)$ are contained in the line $\Re(s) = 1/2$.
\end{itemize}
\end{theorem}

The proof of the Ramanujan conjecture is essentially that of \cite{lomeli2009}. And the proofs of Properties~(ii)--(iv) are those appearing in \cite{ls}. However, they are completed with the results of this article. Basically, we now prove Theorem~\ref{fff} below in a characteristic free way for any global function field.

Let us summarize the results of \cite{lomeli2009} on the globally generic functorial lift for the classical groups.  Let ${\bf G}_n$ be a split classical group of rank $n$ defined over a global function field $k$. The functorial lift of \cite{lomeli2009} takes globally generic cuspidal automorphic representations $\pi$ of ${\bf G}_n(\mathbb{A}_k)$ to automorphic representations of ${\bf H}_N(\mathbb{A}_k)$, where ${\bf H}_N$ is chosen by the following table.

\begin{center}
\begin{tabular}{|c|c|c|} \hline
   ${\bf G}_n$			& ${}^LG_n \hookrightarrow {}^LH_N$							& ${\bf H}_N$ \\ \hline
   ${\rm SO}_{2n+1}$	& ${\rm Sp}_{2n}(\mathbb{C}) \times \mathcal{W}_k \hookrightarrow {\rm GL}_{2n}(\mathbb{C}) \times \mathcal{W}_k$		& ${\rm GL}_{2n}$ \\
   ${\rm  Sp}_{2n}$		& ${\rm SO}_{2n+1}(\mathbb{C}) \times \mathcal{W}_k  \hookrightarrow {\rm GL}_{2n+1}(\mathbb{C}) \times \mathcal{W}_k$	& ${\rm GL}_{2n+1}$ \\
   ${\rm SO}_{2n}$		& ${\rm SO}_{2n}(\mathbb{C}) \times \mathcal{W}_k \hookrightarrow {\rm GL}_{2n}(\mathbb{C}) \times \mathcal{W}_k$		&${\rm GL}_{2n}$ \\
   \hline
\end{tabular}

\smallskip
Table~1
\end{center}

\begin{theorem}\label{fff}
Let ${\bf G}_n$ be a split classical group. Let $\pi$ be a globally generic cuspidal representation of ${\bf G}_n(\mathbb{A}_k)$; $n \geq 2$ if ${\bf G}_n = {\rm SO}_{2n}$. Then, $\pi$ has a functorial lift to an automorphic representation $\Pi$ of ${\bf H}_N(\mathbb{A}_k)$. It has trivial central character and can be expressed as an isobaric sum
\begin{equation*}
   \Pi = \Pi_1 \boxplus \cdots \boxplus \Pi_d,
\end{equation*}
where each $\Pi_i$ is a unitary self-dual cuspidal automorphic representation of ${\rm GL}_{N_i}(\mathbb{A}_k)$ such that $\Pi_i \ncong \Pi_j$ for $i \neq j$. Furthermore, $\Pi_v$ is the local Langlands functorial lift of $\pi_v$ at every place $v$ of $k$. That is, for every $(k_v,\pi_v,\tau_v,\psi_v) \in \mathscr{L}_{\rm loc}(p,{\bf G}_n,{\rm GL}_m)$ there is equality of local factors
\begin{align*}
   \gamma(s,\Pi_v \times \tau_v,\psi_v) &= \gamma(s,\pi_v \times \tau_v,\psi_v) \\
   L(s,\Pi_v \times \tau_v)& = L(s,\pi_v \times \tau_v) \\
   \varepsilon(s,\Pi_v \times \tau_v,\psi_v) &= \varepsilon(s,\pi_v \times \tau_v,\psi_v).
\end{align*}
\end{theorem}
\begin{proof}
Let $\Pi'$ be the automorphic representation of ${\rm GL}_N(\mathbb{A}_k)$ obtained from $\pi$ via the weak functorial lift of Theorem~8.5 of \cite{lomeli2009}. This lift has the property that $\Pi_v'$ is the unramified lift of $\pi_v$ for all $v \notin S$ of \S~8.2.1 of [\emph{loc.\,cit.}]. From the classification of automorphic forms for ${\rm GL}_N$ \cite{js1981}, it is possible to find a globally generic representation $\Pi$ of ${\rm GL}_N(\mathbb{A}_k)$ such that
\begin{equation*}
   \Pi_v \cong \Pi_v'
\end{equation*}
for almost all places $v$ of $k$. It is the unique generic subquotient of an globally induced representation of the form ${\rm Ind}(\Pi_1' \otimes \cdots \otimes \Pi_d')$
with each $\Pi_i'$ a cuspidal representation of ${\rm GL}_{N_i}(\mathbb{A}_k)$. We write this as an isobaric sum
\begin{equation*}
   \Pi =  \Pi_1' \boxplus \cdots \boxplus \Pi_d',
\end{equation*}
The representation $\Pi$ is unitary, what we need to show is that each $\Pi_i'$ is unitary. For this, write $\Pi_i' = \left| \det(\cdot) \right|^{t_i} \Pi_i$ with each $\Pi_i$ unitary and $t_d \geq \cdots \geq t_1$, $t_1 \leq 0$. Because $\Pi$ is unitary, we cannot have $t_1>0$. Consider the quadruple $(k,\pi,\Pi_1,\psi,S) \in \mathscr{L}_{\rm glob}(p,{\bf G}_n,{\rm GL}_{N_1})$, then
\begin{align*}
   L(s,\pi \times \tilde{\Pi}_1) &= L(s,\Pi \times \tilde{\Pi}_1) \\
   						&= \prod_{i} L(s + t_i,\Pi_i \times \tilde{\Pi}_1). \nonumber
\end{align*}
From Theorem~\ref{globholo} we have that $L(s,\pi \times \tilde{\Pi}_1)$ is holomorphic for ${\rm Re}(s) > 1$. However, $L(s + t_1,\Pi_1 \times \tilde{\Pi}_1)$ has a simple pole at $s = 1 - t_1 \geq 1$. This pole carries through to a pole of $L(s,\pi \times \tilde{\Pi}_1)$, which gives a contradiction unless $t_1 = 0$. A recursive argument shows that all the $t_i$'s must be zero.

Also, every time we have $\Pi_i \cong \Pi_j$, we add to the multiplicity of the pole at $s=1$ of the $L$-function
\begin{equation}\label{eq1:fff}
   L(s,\pi \times \tilde{\Pi}_i) = \prod_{j} L(s,\Pi_j \times \tilde{\Pi}_i).
\end{equation}
From Theorem~\ref{globholo}, $L(s,\pi \times \tilde{\Pi}_j)$ has at most a simple pole at $s = 1$. Hence, we must have $\Pi_i \ncong \Pi_j$, for $i \neq j$.

To show that each $\Pi_i$ is self-dual, we argue by contradiction and assume $\Pi_i \ncong \tilde{\Pi}_i$. Notice that for $(k,\pi,\Pi_i,\psi,S) \in \mathscr{L}_{\rm glob}(p,{\bf G}_n,{\rm GL}_{N_i})$, the representation $\sigma = \tilde{\Pi}_i \otimes \tilde{\tau}$ of ${\bf M}(\mathbb{A}_k)$ satisfies $w_0(\sigma) \ncong \sigma$. From Corollary~\ref{Lpoly}, the $L$-function $L(s,\pi \times \tilde{\Pi}_i)$ has no poles. Now, we have that a pole appears on the right hand side of equation~\eqref{eq1:fff} from the $L$-function $L(s,\Pi_i \times \tilde{\Pi}_i)$. This is a contradiction, unless $\Pi_i$ is self-dual.

We now follow the proof 
of Proposition~9.5 of \cite{lsU} to obtain that $\Pi_v$ is the unique local Langlands functorial lift of $\pi_v$ at every place $v$ of $k$. A crucial step is the stability property of $\gamma$-factors under twists by highly ramified characters, Theorem~6.10 of \cite{lomeli2009}, which is available in characteristic two. Lemmas~9.2 and 9.3 of \cite{lsU} indicate how to obtain the desired equality of local factors.
\end{proof}

\end{document}